\documentclass[11pt,leqno]{article}
\usepackage{amssymb,amsthm,amsmath,amscd,amsfonts,bbm,mathrsfs,fullpage, comment}
\usepackage[all]{xy}

 \def\noproof{{\unskip\nobreak\hfill\penalty50\hskip2em\hbox{}%
      \nobreak\hfill$\Box$\parfillskip=0pt%
     \finalhyphendemerits=0\par}}
\def\enddemo{\ifmmode\eqno\Box\else\noproof\vskip0.8truecm\fi}

\newtheorem{theo}{Theorem}[section]

\newtheorem{df}[theo]{Definition}

\newtheorem{lemma/def}[theo]{Lemma/Definition}

\newtheorem{prop}[theo]{Proposition}

\newtheorem{remark}[theo]{Remark}
\newtheorem{remarks}[theo]{Remarks}

\newtheorem{lemma}[theo]{Lemma}

\newcommand{\lra}{\longrightarrow}
\newcommand{\lla}{\longleftarrow}
\newcommand{\hra}{\hookrightarrow}

\DeclareMathOperator{\Ban}{Ban}
\DeclareMathOperator{\Spec}{Spec}

\DeclareMathOperator{\Ind}{Ind}
\DeclareMathOperator{\cInd}{c-Ind}
\DeclareMathOperator{\ab}{ab}

\DeclareMathOperator{\aug}{aug}
\DeclareMathOperator{\fgaug}{fg aug}

\DeclareMathOperator{\rec}{rec}

\DeclareMathOperator{\ord}{ord}

\DeclareMathOperator{\id}{id}

\DeclareMathOperator{\Hom}{Hom}

\DeclareMathOperator{\Ext}{Ext}

\DeclareMathOperator{\End}{End}
\DeclareMathOperator{\St}{St}

\DeclareMathOperator{\fl}{fl}
\DeclareMathOperator{\pr}{pr}

\DeclareMathOperator{\rank}{rank}

\DeclareMathOperator{\tor}{tor}
\DeclareMathOperator{\wcdot}{\, \cdot\, }

\DeclareMathOperator{\Ker}{Ker}

\DeclareMathOperator{\Image}{Im}
\DeclareMathOperator{\Jac}{Jac}

\DeclareMathOperator{\Maps}{Maps}

\DeclareMathOperator{\Frob}{Frob}

\DeclareMathOperator{\Tr}{Tr}

\DeclareMathOperator{\Gal}{Gal}
\DeclareMathOperator{\Norm}{N}
\DeclareMathOperator{\Nrd}{Nrd}
\DeclareMathOperator{\Mod}{Mod}

\DeclareMathOperator{\Ram}{Ram}

\DeclareMathOperator{\GL}{GL}

\DeclareMathOperator{\PGL}{PGL}

\DeclareMathOperator{\ad}{ad}

\DeclareMathOperator{\Ann}{Ann}

\DeclareMathOperator{\an}{an}
\DeclareMathOperator{\Stab}{Stab}

\DeclareMathOperator{\et}{et}

\DeclareMathOperator{\rig}{rig}

\DeclareMathOperator{\JL}{JL}

\DeclareMathOperator{\adm}{adm}

\DeclareMathOperator{\univ}{univ}

\DeclareMathOperator{\cont}{cont}

\DeclareMathOperator{\cl}{cl}

\DeclareMathOperator{\cycl}{cycl}

\DeclareMathOperator{\prolim}{\underset{\lla}{\lim}}
\DeclareMathOperator{\prolimn}{\underset{\underset{n}{\lla}}{\lim}}
\DeclareMathOperator{\prolimm}{\underset{\underset{m}{\lla}}{\lim}}

\newcommand{\io}{{\iota}}
\newcommand{\La}{{\Lambda}}
\newcommand{\la}{{\lambda}}

\newcommand{\fa}{{\mathfrak a}}

\newcommand{\fd}{{\mathfrak d}}

\newcommand{\ff}{{\mathfrak f}}

\newcommand{\fG}{{\mathfrak G}}

\newcommand{\fm}{{\mathfrak m}}

\newcommand{\fn}{{\mathfrak n}}
\newcommand{\fp}{{\mathfrak p}}
\newcommand{\fP}{{\mathfrak P}}
\newcommand{\fq}{{\mathfrak q}}

\newcommand{\fU}{{\mathfrak U}}
\newcommand{\fX}{{\mathfrak X}}

\DeclareMathOperator{\fOrd}{{\mathbb{O}rd}}

\newcommand{\bC}{{\mathbb C}}
\newcommand{\bD}{{\mathbb D}}
\newcommand{\bF}{{\mathbb F}}
\newcommand{\bG}{{\mathbb G}}
\newcommand{\bH}{{\mathbb H}}

\newcommand{\bP}{{\mathbb P}}
\newcommand{\bQ}{{\mathbb Q}}
\newcommand{\bR}{{\mathbb R}}
\newcommand{\bT}{{\mathbb T}}
\newcommand{\bZ}{{\mathbb Z}}
\newcommand{\barQ}{{\overline{\mathbb Q}}}

\newcommand{\bA}{{\mathbb A}}

\newcommand{\barbD}{{\overline{\bD}}}

\newcommand{\barD}{{\overline{D}}}

\newcommand{\barF}{{\overline{F}}}
\newcommand{\barG}{{\overline{G}}}
\newcommand{\barK}{{\overline{K}}}

\newcommand{\barT}{{\overline{T}}}

\newcommand{\barX}{{\overline{X}}}
\DeclareMathOperator{\barep}{\overline{\epsilon}}
\DeclareMathOperator{\barDelta}{\overline{\Delta}}

\newcommand{\cA}{{\mathcal A}}

\newcommand{\cD}{{\mathcal D}}

\newcommand{\cG}{{\mathcal G}}
\newcommand{\cH}{{\mathcal H}}

\newcommand{\cM}{{\mathcal M}}
\newcommand{\cN}{{\mathcal N}}
\newcommand{\cO}{{\mathcal O}}

\newcommand{\cR}{{\mathcal R}}

\newcommand{\cT}{{\mathcal T}}
\newcommand{\cV}{{\mathcal V}}

\newcommand{\sA}{{\mathscr A}}

\newcommand{\sE}{{\mathscr E}}

\newcommand{\sK}{{\mathscr K}}
\newcommand{\sL}{{\mathscr L}}

\newcommand{\sS}{{\mathscr S}}

\newcommand{\wsE}{{\widetilde{\mathscr E}}}

\newcommand{\wE}{\widetilde{E}}
\newcommand{\wG}{\widetilde{G}}

\newcommand{\wV}{\widetilde{V}}

\newcommand{\wcM}{{\widetilde{\mathcal M}}}
\newcommand{\wcO}{{\widetilde{\mathcal O}}}
\newcommand{\wcN}{{\widetilde{\mathcal N}}}

\newcommand{\wK}{\widetilde{K}}

\newcommand{\wchi}{\widetilde{\chi}}
\newcommand{\wpi}{\widetilde{\pi}}

\newcommand{\ep}{{\epsilon}}
\newcommand{\vep}{{\varepsilon}}

\newcommand{\bu}{{\bullet}}

\newcommand{\noi}{\noindent}

\begin{document}
\title{On $\sL$-invariants associated to Hilbert modular forms}

\author{By Michael Spie{\ss}}
\maketitle

\begin{abstract}
Given a cuspidal Hilbert modular eigenform $\pi$ of parallel weight 2 and a nonarchimedian place $\fp$ of the underlying totally real field such that the local component of $\pi$ at $\fp$ is the Steinberg representation, one can associate two types of $\sL$-invariants, one defined in terms of the cohomology of arithmetic groups and the other in terms of the Galois representation associated to $\pi$. We show that the $\sL$-invariants are the same.
\end{abstract}

\tableofcontents

\section{Introduction}

Let $F\subset \barQ$ be a totally real number field, let $D$ be a quaternion algebra over $F$ and let $G/F$ be the algebraic group $D^*/F^*$. Let $d$ be the number of archimedean places of $F$ where $D$ is unramified and let $\fp$ be a fixed 
nonarchimedean place of $F$ that is unramified in $D$ lying above a prime number $p$. Let $\pi =\bigoplus_v' \pi_v$ be a cuspidal automorphic representation of $G(\bA)$ (where $\bA$ denotes the adele ring of $F$) of parallel weight 2 \footnote{By that we mean that the Jacquet-Langlands transfer of $\pi$ to $\PGL_2(\bA)$ corresponds to a cuspidal Hilbert modular eigenform of parallel weight 2} such that the local component $\pi_{\fp}$ is the Steinberg representation of $G(F_{\fp}) \cong \PGL_2(F_{\fp})$. Moreover we fix a finite extension $E/\bQ_p$ with uniformizer $\varpi$ so that $\pi^{\infty} =\bigoplus_{v\nmid \infty}' \pi_v$ is defined over $E$. 

Generalizing a construction of Darmon \cite{darmon}, the author (in the case $D=M_2(F)$) and Gehrmann (for arbitrary $D$) introduced certain $p$-adic numbers called {\it automorphic $\sL$-invariants $\sL_{\ep}(\pi, \psi)$} of $\pi$ (see \cite{spiess1}, \cite{gehrmann1}). They are defined in terms of the group cohomology of $G(F)$. Here $\ep$ is a character of the group of connected components of $\pi_0(G(F\otimes \bR))\cong (\bZ/2\bZ)^d$ and $\psi$ is a continuous homomorphism from the multiplicative group of the local field $F_{\fp}^*$ to the additive group of $E$. 

On the other hand let $\rho_{\pi}: \Gal(\barQ/F)\to \GL_E(V_{\pi})$ be the Galois representation associated to $\pi$. The fact 
that $\pi_{\fp}$ is the Steinberg representation implies that $\rho_{\pi}$ is a direct summand of the Tate module of an abelian variety defined over $F$ with split multiplicative reduction at $\fp$. Hence one can associate an {\it arithmetic} (i.e.\ Mazur-Tate-Teitelbaum type) $\sL$-invariant $\sL(V_{\pi}, \psi)$ to $V_{\pi}$ and $\psi$. Our main result (see Theorem \ref{theo:linvmttautom}) is the equality between automorphic and arithmetic $\sL$-invariants
\begin{equation}
\label{lautarth}
\sL_{\ep}(\pi, \psi) \, =\, \sL(V_{\pi}, \psi).
\end{equation}
In particular $\sL_{\ep}(\pi, \psi)$ is independent of the the sign character $\ep$. For the history of this problem we refer to Remark \ref{remark:mainthm}.

Our proof consists of three steps. If $d=0$ (i.e.\ if $D$ is totally definite) the $\sL$-invariant $\sL_{\ep}(\pi, \psi)$ is essentially Teitelbaum's $\sL$-invariant \cite{teitelbaum} and the equality \eqref{lautarth} can be deduced from $p$-adic uniformization of Shimura curves due to \u{C}erednik (for $F=\bQ$) and Boutot-Zink (for arbitrary $F$). 

If $d=1$ our approach is novel and uses as an essential tool the cohomology theory of {\it $\sS$-varieties} introduced in \cite{spiess2}. It allows us to construct directly\footnote{i.e.\ without deforming $V_{\pi}$ into a $p$-adic family first}
infinitesimal deformations of the Galois representation $V_{\pi}$ involving the automorphic $\sL$-invariant. More precisely since both sides of \eqref{lautarth} are linear in $\psi$ and since both sides are $=1$ if $\psi= v_{\fp}: F_{\fp}^* \to \bZ\subseteq E$ is the normalized valuation at $\fp$ it suffices to see that the vanishing of the left hand side of \eqref{lautarth} implies the vanishing of the right hand side. We will see that the vanishing of \eqref{lautarth} is equivalent to the fact that
\[
\wV_{\pi} \, :=\, \bH_{\varpi-\ad}^1(X_{\barQ}^D; \wE(\Theta_{\psi}), E)_{\pi}
\]
is a non-trivial infinitesimal Galois deformation of $V_{\pi}$. We recall that given an admissible $E$-Banach space representation $W$ of the standard maximal torus $T_{\fp} \cong F_{\fp}^*$ of $\PGL_2(F_{\fp})$ we have constructed in \cite{spiess2} cohomology groups 
\[
\bH_{\varpi-\ad}^n(X^D_{\barQ}; W, E), \qquad n\ge 0
\]
associated to a certain $\sS$-curve $X=X^D$ (i.e.\ an $\fp^{\infty}$-tower of Shimura resp.\ modular curves equipped with a $\PGL_2(F_{\fp})$-action). These are $E$-vector spaces equipped with a continuous $\Gal(\barQ/F)$-action as well as a spherical Hecke action. So we can consider in particular the $E$-Banach representation 
$W= \wE(\Theta_{\psi})$ where $\wE=E[\vep]$ denote the dual numbers and $\Theta_{\psi}: F_{\fp}^*\to \wE$ the character $\Theta_{\psi}(x) = 1+ \psi(x) \vep$. In (\cite{spiess2}, \S 5.6 and 5.7),  we have shown that $\wV_{\pi}$ admits a two-step filtration and that the action of $\Gal(\barQ_p/F_{\fp})$ on the associated graded modules factors through the maximal abelian quotient $\Gal(\barQ_p/F_{\fp})^{\ab}$ and can be described explicitly. This enables us to follow the argument of Greenberg and Stevens  (\cite{grstevens}, proof of Thm.\ 3.14) to deduce $\sL(V_{\pi}, \psi)=0$. 

In the case $d>1$ we reduce the proof of \eqref{lautarth} to the case $d\le 1$ by proving a certain Jacquet-Langlands functoriality for $\sL_{\ep}(\pi, \psi)$. More precisely if $\barD$ denotes the quaternion algebra over $F$ which is ramified at the same nonarchimedean places as $D$ and at all infinite places except possibly one and if $\pi'$ denotes the Jacquet-Langlands transfer of $\pi$ to $\barG(\bA)$ (with $\barG = \barD^*/F^*)$ then we will show (see Theorem \ref{theo:linvjl})
\begin{equation}
\label{ljaclang}
\sL_{\ep}(\pi, \psi) = \sL(\pi', \psi).
\end{equation}
Let $\bar{d}\in \{0,1\}$ denote the number of archimedean places ramified in $\barD$.
As a key tool in proving \eqref{ljaclang} we study the cohomology groups 
\[
\bH_{\varpi-\ad}^d((X^D)^{\an}; \fOrd_{\varpi-\ad}^{\bar{d}}((X^{\barD})^{\an}, \cO)_E, E)
\]
The cohomology group $\fOrd_{\varpi-\ad}^{\bar{d}}((X^{\barD})^{\an}, \cO)_E$ has been introduced in (\cite{spiess2}, \S 4.1). It is an admissible $E$-Banach space representation of $T_{\fp}$. 

\section{Admissible Banach space representations and extensions of the Steinberg representations}
\label{section:borelind}

\paragraph{Notation and preliminary remarks} Let $R$ be a commutative noetherian ring and let $\cG$ be a locally profinite group. Recall that a left $R[\cG]$-module $W$ is called discrete (or smooth) if $W$ is discrete as a $\cG$-module, i.e.\ if the stabilizer $\Stab_{\cG}(w)=\{g\in \cG\mid gw=g\}$ is open in $\cG$ for every $w\in W$. A discrete $R[\cG]$-module $W$ is called admissible if $W^U$ is a finitely generated $R$-module for every open subgroup $U$ of $\cG$. 

Throughout this section $F$ denotes a $p$-adic field, i.e.\ a finite extension of $\bQ_p$. We let 
$v_F: F\to \bZ\cap\{+\infty\}$ denote the normalized valuation of $F$, $\cO_F$ its valuation ring, $\fp$ its valuation ideal and $U_F = \cO_F^*$ the group of units.

Let $G=\PGL_2(F) = \GL_2(F)/Z$ and let $\pr: \GL_2(F)\to G$ be the projection. Let $B$ be the standard Borel subgroup of $G$ and let $B = T N$ be the Levi decomposition of $B$, so $T$ consists of the diagonal matrices (modulo the center $Z$ of $\GL_2(F)$) and $N$ consists of the upper triangular matrices with $1$ as diagonal entries (mod $Z$). In the following we are often going to identify $T$ with the onedimensional split torus $F^*$ via the isomorphism 
\begin{equation}
\label{splittorus}
\delta: F^*\lra T, \,\, x\mapsto \delta(x)=\left(\begin{matrix} x & \\ & 1\end{matrix}\right) \mod Z
\end{equation}
and $N$ with the additive group $F$ via the isomorphism
\begin{equation}
\label{unipotent}
n: F\lra N, \,\, y\mapsto n(y)=\left(\begin{matrix} 1 & y\\  0 & 1\end{matrix}\right) \mod Z.
\end{equation}
so any $b\in B$ can be written uniquely in the form $b=\delta(x)\cdot n(y)$ with $x\in F^*$, $y\in F$. We denote by $T_0$ the maximal compact open subgroup of $T$, so $T_0$ corresponds to $U_F$ under \eqref{splittorus}. We let $N_0$ be the subgroup of $N$ hat corresponds to the $\cO_F$ under the isomorphism \eqref{unipotent}, i.e.\ $N_0=\{ n(y)\mid y\in \cO_F\}$ and put $T^+= \{t\in T\mid N_0^t= t N_0 t^{-1} \subseteq N_0\}$.

For an integer $n\ge 0$ we let $K(\fp^n) = \Ker(\GL_2(\cO_F) \to \GL_2(\cO_F/\fp^n))$ and define $K(n)$ to be the image of $K(\fp^n)$ under the projection  $\pr:\GL_2(F)\to G$. For a closed subgroup $H$ of $T^0$ we put 
\begin{equation}
\label{congruence3}
K_H(n) \, =\, K(n) H N_0.
\end{equation}
 
We fix another finite field extension $E/\barQ_p$ with valuation ring $\cO$, unifomizer $\varpi\in \cO$ and residue field $k=\cO/(\varpi)$. For $m\ge 1$ we put $\cO_m =\cO/(\varpi^m)$. More generally for an $\cO$-module $N$ we denote its torsion submodule by $N_{\tor}$ its the maximal torsionfree quotient by $N_{\fl}$ and we put $N_E= N\otimes_{\cO} E$. We also set $N_m= N\otimes_{\cO}\cO_m$ and $N[\varpi^m] = \Hom_{\cO}(\cO_{m}, N)$ for $m\ge 1$. Multiplication by 
$\varpi$ and $\varpi^m$ induces an exact sequence 
\begin{equation}
\label{kercoker}
0 \lra N[\varpi]\lra N[\varpi^{m+1}]\lra N[\varpi^m]\lra N_1\lra N_{m+1}\lra N_m \lra 0 
\end{equation} 
for every $m\ge 1$. 
In this section we will consider certain representation of $G$ and $T$ on $\cO$-modules or $E$-vector spaces. 

We briefly review the notion of $\varpi$-adically admissible representations and admissible Banach space representations of $T$ and the relation to finitely generated augmented $T$-modules (for further details we refer to \cite{spiess2}, \S 3.4). Then we are going to recall the construction of certain canonical extensions of the Steinberg representation.

\paragraph{$\varpi$-adically admissible representations of $T$} We fix a closed subgroup $H$ of $T^0$ and put $\barT = T/H$ and $\barT^0 = T^0/H$. We briefly review the notion of $\varpi$-adically continuous and $\varpi$-adically admissible $\cO[\barT]$-modules (\cite{emerton}, \S 2.4). An $\cO[\barT]$-module $W$ is called $\varpi$-adically continuous if (i) $W$ is $\varpi$-adically complete and separated, (ii) $W_{\tor}$ is of bounded exponent (i.e.\ $\varpi^m W=0$ for $m\ge 1$ sufficiently large) and (iii) $W_m$ is a discrete $\cO_m[\barT]$-module for every $m\ge 1$. A $\varpi$-adically admissible $\cO[\barT]$-module $W$ is a $\varpi$-adically continuous $\cO[\barT]$-module $W$ such that $W_1$ is an 
admissible $k[\barT]$-module. The exactness of the sequence \eqref{kercoker} implies that $W_m$ is an admissible $\cO_m[\barT]$-module for every $m\ge 1$. The full subcategory of the category of $\cO[\barT]$-modules consisting of $\varpi$-adically admissible $\cO[\barT]$-modules will be denoted by 
$\Mod_{\cO}^{\varpi-\adm}(\barT)$. It is an abelian category. 

We recall also recall the notion of an augmented $\cO[\barT]$-module (see \cite{emerton}, \S 2). 
For an open subgroup $U$ of $\barT$ we consider the $\cO$-algebra
\begin{equation*}
\label{iwasawa1}
\Lambda_{\cO}(U) \, =\, \prolim_{U'} \cO[U/U']
\end{equation*} 
where $U'$ runs over all open compact subgroups of $U$. If $U$ itself is compact then $\La_\cO(U)= \cO[\![U]\!]$ is the usual completed $\cO$-group algebra of $U$. Note that we have $\Lambda_{\cO}(\barT)  =\cO[\![\barT_0]\!][t_0^{\pm 1}]$
where $t_0\in T$ corresponds to a uniformizer in $F$ under \eqref{splittorus}.
In particular the ring $\Lambda_{\cO}(\barT)$ is noetherian.

A $\Lambda_{\cO}(\barT)$-module $L$ is called an {\it augmented $\cO[\barT]$-module}. The category of augmented $\cO[\barT]$-modules will be denote by $\Mod_{\cO}^{\aug}(\barT)$. An augmented $\cO[\barT]$-module $L$ is called finitely generated if $L$ is finitely generated as $\Lambda_{\cO}(U)$-modules (with respect to the canonical embedding $\Lambda_{\cO}(U)\hra \Lambda_{\cO}(\barT)$) 
for some (equivalently, any) compact open subgroup $U$ of $\barT$. The full subcategory of $\Mod_{\cO}^{\aug}(\barT)$ of finitely generated augmented ${\cO}[\barT]$-modules will be denoted by $\Mod_{\cO}^{\fgaug}(\barT)$. There exists a natural (profinite) topology on every objects $L$ of $\Mod_{\cO}^{\fgaug}(\barT)$ (see \cite{emerton}, Prop.\ 2.1.3) such that the action $\Lambda_{\cO}(\barT) \times L \to L$ is continuous. This topology is called the {\it canonical topology}.

\begin{remark}
\label{remark:padicadm}
\rm Let $\psi: F^*\to \cO$ be a continuous homomorphism (i.e.\ we have $\psi(xy) = \psi(x) + \psi(y)$ for all $x,y\in F^*$) and let
$\wcO = \cO[\vep] = \cO[X]/(X^2)$, $\vep := X + (X^2)$ be the $\cO$-algebra of dual numbers. The character 
\begin{equation}
\label{extstein1}
\Theta_{\psi}: F^*\lra \wcO^*,\,\, \Theta_{\psi}(x)=  1+ \psi(x)\ep.
\end{equation}
is again continuous. It induces an $F^*$- hence via \eqref{splittorus} also a $T$-action on $\wcO$. We denote the resulting $\cO[T]$-module by
$\wcO(\Theta_{\psi})$. It is clearly $\varpi$-adically admissible. Note that $\wcO(\Theta_{\psi})$ is an admissible $\cO[T]$-module if and only if $\Ker(\psi)$ is open in $F^*$ (this holds if and only if $\psi$ is multiple $c\cdot v_F$, $c\in \cO$ of the normalized valuation $v_F$ of $F$). \enddemo
\end{remark}

\paragraph{Banach space representations} Recall that an $E$-Banach space representation of $\barT$ is an $E$-Banach space $V=(V, \| \cdot \|)$ together with a continuous $E$-linear action $\barT\times V \to V, (t,v)\mapsto t\cdot v$. An $E$-Banach space representation $V$ of $\barT$ is called admissible if there exists an open and bounded $\cO[\barT]$-submodule $W\subseteq V$ such that the $U$-invariant $(W/V)^U$ of the quotient $V/W$ are an $\cO$-module of cofinite type for every open subgroup $U$ of $\barT$ (i.e.\ the Pontrjagin dual $\Hom((W/V)^U, E/\cO)$ is a finitely generated $\cO$-module).\footnote{Note that this condition implies that we can choose the norm $\|\wcdot \|$ on $V$ so that $(V, \| \cdot \|)$ is a unitary Banach space representation of $\barT$, i.e.\ we have $\|t\cdot v\|=\|v\|$ for every $t\in \barT$ and $v\in V$.}

The category of admissible $E$-Banach space representations of $\barT$ will be denoted by $\Ban_E^{\adm}(\barT)$. It is an abelian category. This can be easily deduced from the duality theorem below or from the fact that $\Ban_E^{\adm}(\barT)$ is equivalent to the localized category $\Mod_{\cO}^{\varpi-\adm}(\barT)_E$. Recall that for an $\cO$-linear additive category $\sA$ the $E$-linear additive category $\sA_E$ has the same objects as $\sA$ whereas the morphisms are given by 
$\Hom_{\sA_E}(A, B) = \Hom_{\sA}(A, B)\otimes_{\cO} E$. The functor
\begin{equation}
\label{locbanach2}
\Mod_{\cO}^{\varpi-\adm}(T)_E\lra \Ban_E^{\adm}(T), \quad W\mapsto (W_E, \| \wcdot\|)
\end{equation} 
is an equivalence of categories. Here for $W\in \Mod_{\cO}^{\varpi-\adm}(T)$ we let $\|\wcdot\|$ be the norm on $V=W_E$, so that  
$\,\Image(W\to V, w\mapsto w\otimes 1)$ is the unit ball $\{v\in V\mid \|v\|\le 1\}$ in $V$. 

Next we review the duality theorem. A $\Lambda_{\cO}(\barT)_E$-module $M$ will be called an augmented $E[\barT]$-module. Again, $\Mod_E^{\aug}(\barT)$ denotes the category of augmented $E[\barT]$-modules. An augmented $E[\barT]$-module $M$ is called finitely generated if there exists a $\La_{\cO}(\barT)$-submodule $L$ with $L\in \Mod_{\cO}^{\fgaug}(\barT)$ and $L_E = M$. We equip $M$ with the topology induced by the canonical topology on $L$, i.e.\ $M$ is a topological vector space and the inclusion is $L\hra M$ is open and continuous (hence $M$ is locally compact). 
As before $\Mod_E^{\fgaug}(\barT)$ denotes the full category of $\Mod_E^{\aug}(\barT)$ of finitely generated augmented $E[\barT]$-modules. 

For an $E$-Banach space $V=(V, \| \wcdot\|)$ we define $\cD(V) = V'$ to be the dual space equipped with the weak topology.
It follows immediately from (\cite{schneider-teitelbaum}, Thm.\ 3.5) that the functor
\begin{equation}
\label{pontrajagin4}
\cD: \Ban_E^{\adm}(\barT)\lra \Mod_E^{\fgaug}(\barT), \,\, V\mapsto \cD(V)= V'
\end{equation} 
is an anti-equivalence of $E$-linear abelian categories. Its quasi-inverse is given by
\begin{equation*}
\label{pontrajagin5}
\Mod_E^{\fgaug}(\barT) \lra  \Ban_E^{\adm}(\barT), \, \,\, M\mapsto \Hom_{E, \cont}(M, E).
\end{equation*} 
For $V\in \Ban_E^{\adm}(\barT)$ one can show that the $E[\barT]$-action on $V$ extends naturally to a $\La_{\cO}(\barT)_E$-action and that the functor \eqref{pontrajagin4} is $\La_{\cO}(\barT)_E$-linear.

\begin{remarks}
\label{remarks:locdualpply}
\rm (a) Let $V_1$, $V_2\in \Ban_E^{\adm}(\barT)$ and assume that $\dim_E(V_1) <\infty$. The duality theorem implies that the $E$-vector space $\Hom_{E[\barT]}(V_1, V_2)$ is finite-dimensional as well. In fact any $\barT$-equivariant homomorphism $V_1\to V_2$ is automatically continuous so we have 
\begin{equation}
\label{pontrajagin5a}
\Hom_{E[\barT]}(V_1, V_2)\, \cong \, \Hom_{\Mod_E^{\fgaug}(\barT)}(\cD(V_2), \cD(V_1)).
\end{equation}
If $U$ is a compact open subgroup of $\barT$ then the $\Lambda_{\cO}(U)_E$-module \eqref{pontrajagin5a} is finitely generated, hence it has finite length since $\cD(V_1)$ has finite length (because of $\dim_E(\cD(V_1))= \dim_E(V_1) <\infty$). It is therefore finite-dimensional as an $E$-vector space. 

\noi (b) Let $\psi: F^*\to E$ be a continuous character and let
$\wE = E[\vep]$ be the $E$-algebra of dual numbers. Since the image of $\psi$ is bounded in $E$ the image of the character
\begin{equation}
\label{extstein1a}
\Theta_{\psi}: F^*\lra \wE^*,\,\, \Theta_{\psi}(x)=  1+ \psi(x)\ep.
\end{equation}
is bounded in $\wE$. Therefore similar to Remark \ref{remark:padicadm} we obtain an admissible $E$-Banach space representation $\wE(\Theta_{\psi})$ of $T$. \enddemo
\end{remarks}

\paragraph{Extensions of the Steinberg representation} For a $\varpi$-adically continuous $\cO[T]$-module $W$ the $\varpi$-adically continuous parabolic induction of $W$ is defined by 
\begin{equation*}
\label{parindcont}
\Ind_{B, \cont}^G W \,= \, \{\Phi\in C_{\cont}(G, W)\mid \Phi(tng) = t\Phi(g) \,\,\forall t\in T, n\in N, g\in G\}
\end{equation*}
where $C_{\cont}(G, W)$ denotes the $\cO$-module of maps $G \to W$ that are continuous with respect to the $\varpi$-adic topology on $W$. The $G$-action on $\Ind_{B, \cont}^G W$ is induced by right 
multiplication. By (\cite{emerton}, Lemma 4.1.3) we have 
\begin{equation*}
\label{parindcont2}
\Ind_{B, \cont}^G W\,  = \, \prolimm \Ind_B^G W_m.
\end{equation*}
Moreover, for a $\varpi$-adically admissible $\cO[T]$-module $W$ the $\cO[G]$-module $\Ind_{B, \cont}^G W$ is $\varpi$-adically admissible as well and the functor 
\begin{equation*}
\label{parindcont3}
\Ind_{B, \cont}^G : \Mod_{\cO}^{\varpi-\adm}(T) \to \Mod_{\cO}^{\varpi-\adm}(G)
\end{equation*}
is exact (see \cite{emerton}, Prop.\ 4.1.5 and Prop.\ 4.1.7; the category $\Mod_{\cO}^{\varpi-\adm}(G)$ of $\varpi$-adically admissible $\cO[G]$-modules is defined in the same way as $\Mod_{\cO}^{\varpi-\adm}(T)$; see \cite{emerton}, Def.\ 2.4.7).

Recall that the Steinberg representation $\St_{G}(R)$ for any ring $R$ is the admissible $R[G]$-module defined by the short exact sequence 
\begin{equation*}
\label{steinberg}
\begin{CD}
0@>>> R(0) @>(\star)>> \Ind_{B}^G R(0) @>\pr >> \St_{G}(R) @>>> 0.
\end{CD}
\end{equation*}
Here $R(0)$ denotes $R$ with the trivial $G$-action and the map $(\star)$ is given by sending $x\in R$ to the constant map $G\to R, g\mapsto x$. We may also consider the continuous Steinberg representation $\St_{G, \cont}(\cO)$ defined by the sequence 
\begin{equation*}
\label{steinbergcont}
\begin{CD}
0@>>> \cO(0) @>(\star)>> \Ind_{B, \cont}^G \cO(0) @>\pr >> \St_{G, \cont} @>>> 0.
\end{CD}
\end{equation*}

\begin{remark}
\label{remark:steinbergcont}
\rm The continuous Steinberg representation is a $\varpi$-adically admissible $\cO[G]$-module. In fact we have $\St_{G, \cont}(\cO)=\prolimm \St_{G, \cont}(\cO)_m
$ and $\St_{G, \cont}(\cO)_m = \St_{G}(\cO_m)$ for all $m\ge 1$. \enddemo
\end{remark}

We now recall the construction of a certain extension of $\St_{G, \cont}$ associated to a continuous homomorphism $\psi: F^*\to \cO$ (see \cite{spiess1}, \S 3.7). Let $\wcO(\Theta_{\psi})$ be the $\varpi$-adically admissible
$\cO[T]$-module defined in Remark \ref{remark:padicadm}. Multiplication by $\vep$ induces a short exact sequence of $\varpi$-adically admissible $\cO[T]$-modules
\begin{equation*}
\label{extstein6}
0\lra \cO(0) \stackrel{\alpha}{\lra} \wcO(\Theta_{\psi})\stackrel{\beta}{\lra} \cO(0) \lra 0
\end{equation*}
(i.e.\ $\alpha\circ \beta$ is multiplication with $\vep$). By applying $\Ind_{B, \cont}^G$ we obtain a short exact sequence of $\varpi$-adically admissible $\cO[G]$-modules
\begin{equation}
\label{extstein7}
0\lra \Ind_{B, \cont}^G \cO(0)  \lra \Ind_{B, \cont}^G \wcO(\Theta_{\psi}) \lra \Ind_{B, \cont}^G \cO(0) \lra 0.
\end{equation}
By taking the pull-back of \eqref{extstein7} with respect to the map $(\star)$ (for $R=\cO$) we obtain a short exact sequence 
\begin{equation}
\label{extstein8}
0\lra \Ind_{B, \cont}^G \cO(0) \lra \wsE(\psi)  \lra \cO(0) \lra 0
\end{equation}
i.e.\ $\wsE(\psi)$ is the $\cO[G]$-submodule of $\Ind_{B, \cont}^G \wcO(\Theta_{\psi})$ given by $\Phi\in \Ind_{B, \cont}^G \wcO(\Theta_{\psi})$ so that $\Phi \!\!\mod \vep$ is a constant map $G(F) \to \cO$.

\begin{df}
\label{df:extstein}
The ($\varpi$-adically admissible) $\cO[G]$-module $\sE(\psi)$ is defined as the push-out of \eqref{extstein8} under the projection 
$\pr: \Ind_{B, \cont}^G \cO(0) \to \St_{G,\cont}(\cO)$ so that we have a short exact sequence of $\cO[G]$-modules
\begin{equation}
\label{extstein9}
0\lra \St_{G, \cont}(\cO) \lra \sE(\psi)  \lra \cO(0) \lra 0.
\end{equation}
\end{df}

We need the following Lemma.

\begin{lemma}
\label{lemma:extsteinfunc}
The map 
\begin{equation}
\label{extstein10}
\Hom_{\cont}(F^*, \cO) \lra \Ext_{\Mod_{\cO}^{\varpi-\adm}(G)}(\cO(0), \St_{G, \cont}(\cO)), \,\, \psi\mapsto [\sE(\psi)]_{\sim}
\end{equation}
is a homomorphism of $\cO$-modules.
\end{lemma}

For the proof of the additivity of \eqref{extstein10} see (\cite{gehrstein}, Lemma 1 (c)).

\begin{remark}
\label{remark:extsteindisc}
\rm For $\psi= v_F$ the sequence \eqref{extstein6} consists of admissible $\cO[T]$-modules. Therefore we may apply usual parabolic induction so we obtain the sequence 
\begin{equation}
\label{extstein7a}
0\lra \Ind_{B}^G \cO(0)  \lra \Ind_{B}^G \wcO(\Theta_{\psi}) \lra \Ind_{B}^G \cO(0) \lra 0.
\end{equation}
By mimicking the construction of the extension \eqref{extstein9} we thus obtain a sequence of admissible $\cO[G]$-modules
\begin{equation}
\label{extstein9a}
0\lra \St_{G}(\cO) \lra \sE_0  \lra \cO(0) \lra 0.
\end{equation}
We remark that the $\cO[G]$-module $\sE_0$ admits the following resolution 
\begin{equation}
\label{extstein9b}
\begin{CD}
0 @>>> \cInd^{G}_{K(0)} \cO(0) @> T-(q+1)\id >> \cInd^{G}_{K(0)} \cO(0)@>>> \sE_0@>>> 0
\end{CD}
\end{equation}
This fact will be used in next section. Note that the sequence \eqref{extstein9} (for $\psi=v_F$) is the push-out of \eqref{extstein9a} with respect to the natural embedding $\St_G(\cO)\hra \St_{G, \cont}(\cO)$. \enddemo
\end{remark} 

\paragraph{Universal extension of the Steinberg representation} Following \cite{gehrstein, berggehr} the definition of the extension \eqref{extstein9} can be refined as follows. 
Let $A$ be a locally profinite abelian group and let $\psi: F^* \to A$ be a continuous homomorphism. There exists a canonical extension 
\begin{equation}
\label{extstein9c}
0\lra \St_{G, \cont}(A) \lra \sE'(\psi)  \lra \bZ(0) \lra 0
\end{equation}
associated to $\psi$ defined as follows. Let $\cR = \bZ+ \ep A$ be the topological ring of formal sums $m + \vep a$ with $m\in \bZ$, $a\in A$. Addition and multiplication is given by $(m_1 + \vep a_1) + (m_2 + \vep a_2)= (m_1 + m_2) + \vep(a_1 + a_2)$ and 
$(m_1 + \vep a_1) \cdot (m_2 + \vep a_2)= (m_1 \cdot m_2) + \vep(m_ 2 a_1 + m_1 a_2)$ respectively (the topology is the product topology with $\bZ$ being equipped with the discrete topology). Then $x\mapsto \Theta_{\psi}(x) =  1+ \psi(x)\ep$ can be viewed as a continuous homomorphism $\Theta_{\psi}: F^*\to \cR^*$, so as before we can consider the continuous parabolic induction $\Ind_{B, \cont}^G \cR[\Theta_{\psi}]$ of the $\cR[T]$-module $\cR[\Theta_{\psi}]$. By applying the functor $\Ind_{B, \cont}^G$ to the short exact sequence 
\begin{equation*}
\label{extstein9d}
\begin{CD}
0@>>> A(0) @> a \mapsto \vep a >> \cR[\Theta_{\psi}]@> m + \vep a \mapsto m >> \bZ(0) @>>> 0
\end{CD}
\end{equation*}
we obtain a short exact sequence of $G$-modules
\begin{equation*}
\label{extstein9f}
0\lra \Ind_{B, \cont}^G A(0)  \lra \Ind_{B, \cont}^G \cR(\Theta_{\psi}) \lra \Ind_B^G \bZ(0) \lra 0.
\end{equation*}
Again by pulling back this sequence with respect to the canonical inclusion $\bZ(0) \hra \Ind_B^G \bZ(0)$ and then taking the push-out with respect to the projection $\Ind_{B, \cont}^G A(0) \to \St_{G, \cont}(A) := \left(\Ind_{B, \cont}^G A(0)\right)/A(0)$ yields \eqref{extstein9c}. In the special case where $\psi$ is the identity $\id: F^* \to F^*$ we obtain the universal extension of the Steinberg representation
\begin{equation}
\label{extsteinuniv}
0\lra \St_{G, \cont}(F^*) \lra \sE_{\univ}  \lra \bZ(0) \lra 0
\end{equation}
(where $\sE_{\univ}:= \sE'(\id)$).

\section{Automorphic and arithmetic $\sL$-Invariants}
\label{section:linv}

In this section we mostly recall the definition of automorphic and  arithmetic $\sL$-invariants and state our main result. 
To begin with we introduce some notation. 

For the remainder of this paper $F$ denotes a totally real number field with ring of algebraic integers $\cO_F$. We fix a nonarchimedian place $\fp$ of $F$ lying above a prime number $p$. We denote by $\bA$ (resp.\ $\bA^{\fp}$, resp.\ $\bA_f$, resp.\ $\bA_f^{\fp}$) the ring of adeles of $F$ (resp.\ ring of prime-to-$\fp$, resp.\ finite, resp.\ finite prime-to-$\fp$ adeles). We let $S_{\infty}$ denote the set of archimedean places of $F$ and put $F_{\infty} = F\otimes_{\bQ} \bR$. For a nonarchimedean place $\fq$ of $F$ we let $\cO_{\fq}$ denote the valuation ring in $F_{\fq}$ and we put $U_{\fq}^{(0)} = U_{\fp} = \cO_{\fq}^*$ and $U_{\fq}^{(n)} =1 +\fq^n\cO_{\fq}$ and for $n\ge 1$. For a place $v$ of $F$ and an algebraic group $\cG/F$ we often write $\cG_v$ for $\cG(F_v)$.

Let $D$ be a quaternion algebra over $F$, let $\wG=D^*$ (viewed as an algebraic group over $F$), let $Z\cong \bG_m$ be the center of $\wG$ and put $G = \wG/Z$ (thus $G={\PGL_2}{_{/F}}$ if $D=M_2(F)$ is the algebra of $2\times 2$ matrices). We let $\Ram_D$ be the set of (archimedian or nonarchimedian) places of $F$ that are ramified in $D$. We assume that $\fp$ does not lie in $\Ram_D$ so that $G_{\fp} = \PGL_2(F_\fp)$. We denote by $B_{\fp}$ the standard Borel subgroup of $G_{\fp}$ and by $T_{\fp}$ its maximal torus. We denote by $\Sigma$ the set of archimedean places of $F$ that split $D$ (i.e.\ $\Sigma = S_{\infty}\setminus \Ram_D$) and we put $d=\sharp{\Sigma}$. If $D$ is not totally ramified then we choose an ordering $\Sigma = \{\sigma_1, \ldots, \sigma_d\}$ of the places in $\Sigma$. We let $\fd$ be the ideal of $\cO_F$ that is the product of the primes which are ramified in $D$, so that $\Ram_D = (S_{\infty}-\Sigma) \cup \{\fq\mid \fd\}$. 

For $v\in S_{\infty}$ we denote by $G_{v, +}$ the connected component of $1$ in $G_v$. Thus $G_{v, +}=(\wG_v)_+/Z_v$ where $(\wG_v)_+$ is the subgroup of elements $g\in \wG_v$ with $\Nrd(g) >0$. We also put $G(F)_+ = G(F) \cap \prod_{v\mid \infty} G_{v, +}$ and define 
\begin{equation*}
\label{pi0ginfty}
\Delta := \, G(F)/G(F)_+ \, \cong \,  \{\pm1\}^d.
\end{equation*}
A compact open subgroup $K_f^{\fp}$ of $G(\bA_f^{\fp})$ will be called a (prime-to-$\fp$) level.
We consider in particular levels of the form $K_0(\fn)^{\fp}$. We recall the definition. Let $\fn\ne (0)$ be an ideal of $\cO_F$ that is relatively prime to $\{\fp\} \cup \Ram_D$. Let $\cO_D$ be an Eichler order of level $\fn$ in $D$ (if $D= M_2(F)$ then we choose $\cO_D$ to be the subalgebra $M_0(\fn)\subseteq M_2(\cO_F)$ of matrices that are upper triangular modulo $\fn$).  For a nonarchimedean place $\fq$ of $F$ we put $\cO_{D, \fq} = \cO_D\otimes_{\cO_F} \cO_{\fq}$ and define $K_0(\fn)^{\fp}$ to be the image of 
\begin{equation}
\label{level}
\wK_0(\fn)^{\fp} \, =\, \prod_{\fq\nmid \fp \infty} \cO_{D, \fq}^*
\end{equation}
under the projection $\wG(\bA_f^{\fp})\to G(\bA_f^{\fp})$.

Throughout this section we consider a fixed cuspidal automorphic representation $\pi = \otimes_v \, \pi_v$ of $G(\bA_F)$ of parallel weight $2$, i.e.\ it has following properties
\medskip

\noi $-$ $\pi_v$ is the discrete series representation of $G_v$ of weight $2$ for all $v\in S_{\infty}-\Ram_D$,
\medskip

\noi $-$ $\pi_v$ is the trivial representation of $G_v$ for all $v\in S_{\infty}\cap \Ram_D$.
\medskip

\noi For simplicity we also assume that $\pi_v$ is a onedimensional representation of $G_v$ for all $v\in \Ram_D-S_{\infty}$.

\paragraph{Automorphic $\sL$-invariants} We put $\pi_f= \pi^{\infty} = \otimes_{v\nmid \infty} \, \pi_v$ and 
$\pi_0= \pi^{\fp\fd, \infty}=  \otimes_{v\nmid \fp\fd\infty} \, \pi_v$ and assume that its conductor is $\ff(\pi_0)= \fn$. In order to associate to $\pi$ automorphic and arithmetic $\sL$-invariants we need to assume that $\pi$ is of split multiplicative type at $\fp$, i.e.\ we assume that the following holds
\begin{equation}
\label{steinbergaut}
\pi_\fp \,\cong\, \St_{G_{\fp}}(\bC).
\end{equation}
This implies in particular that the conductor of $\pi$ is $\ff(\pi) = \fp\fn$. 

The $G(\bA_f)$-representation $\pi_f$ can be defined over an algebraic number field $\bQ_{\pi}\subseteq \barQ \subseteq \bC$, i.e.\ $\pi_f$ (and therefore also $\pi_0$) admits a $\bQ_{\pi}$-structure. The spherical (prime-to $p$) Hecke algebra $\bT = \bT^{S_0}=\bZ[T_{\fq}|\fq\nmid p\fd\fn]$ (where $S_0 = \{\fq \mid p\fd\fn\}$) acts on the onedimensional vector space $\pi_0^
{K_0(\fn)^{\fp}}$ via a ringhomomorphism $\la_{\pi}: \bT\lra \bC$ (the Hecke eigenvalue homomorphism)
whose image is contained in $\cO_{\bQ_{\pi}}$. We choose a place $\fP$ of $\bQ_{\pi}$ above $p$, or equivalently, we choose an embedding $\bQ_{\pi}\hra \bC_p$ and a subfield $E$ of $\bC_p$ that is finite over $\bQ_p$ and contains the image of $\bQ_{\pi}$ (we may choose $E$ to be the completion of $\bQ_{\pi}$ at $\fP$). Let $\cO$ be the valuation ring of $E$ and let $\varpi\in \cO$ be a prime element. Then $\la_{\pi}$ can be viewed as a ringhomomorphism
\begin{equation*}
\label{eigenvalue}
\la_{\pi}: \bT_{\cO}\lra \cO.
\end{equation*}
For a $\bT_{\cO}$-module $N$ we denote by $N_{\pi}$ the localization of $N$ with respect to $\ker(\la_{\pi})$. We will also consider $\bT$-modules $N$ with an additional $\Delta$-action. If $\ep:\Delta\to \{\pm \}$ is a character then the pair $(\la_{\pi}, \ep)$ defines a homomorphism $(\la_{\pi}, \ep): \bT[\Delta]\to \cO$ and we define $N_{\pi, \ep}$ to be the localization of $N$ with respect to its kernel. 

As in \cite{spiess1} for a ring $R$, a level $K_f^{\fp}$ and an $R[G_{\fp}]$-module $M$ we let $\cA_R(M, K_f^{\fp}; R)$ denote the $R[G(F)]$-module of maps $\Phi: M\times G(\bA_f^{\fp})/K_f^{\fp}  \to R$
that are homomorphism of $R$-modules in the first component. The $G(F)$-action is given by $(\gamma \Phi)(m, gK_f^{\fp}) \, =\, \Phi(\gamma^{-1} m, \gamma^{-1}gK_f^{\fp})$
for $\Phi\in \cA_R(M, K_f^{\fp}; R)$, $\gamma\in G(F)$, $m\in M$ and $g\in G(\bA_f^{\fp})$. The cohomology groups
\begin{equation}
\label{derham}
H^{\bu}(G(F)_+, \sA_R(M, K_f^{\fp}; R)) 
\end{equation}
where considered in \cite{spiess1}. Since $\cA_R(M, K_f^{\fp}; R)$ is an $R[G(F)]$-module the groups \eqref{derham} are naturally $R[\Delta]$-modules. There is also a canonical action of the Hecke algebra $R[G(\bA_f^{\fp})//K_f^{\fp}]$ on \eqref{derham} commuting with the $\Delta$-action. 

The automorphic $\sL$-invariants associated to $\pi$ are defined in terms of the cohomology groups
\begin{equation*}
\label{derham1}
H^n(G(F)_+, \cA_{\cO}(M, K_0(\fn)^{\fp}, \cO))_E
\end{equation*}
for $n=d$, $M =\St_{G_{\fp}, \cont}(\cO)$ and the extension classes \eqref{extstein9} (compare \cite{spiess1}, 6.1 and \cite{gehrmann1}, 2.5). They carry a canonical $\bT_E[\Delta]$-module structure. In the following we will abbreviate 
\begin{equation*}
\label{derham2}
\cH^n(M)\colon = H^n(G(F)_+, \cA_{\cO}(M, K_0(\fn)^{\fp}; \cO))_E.
\end{equation*}
We recall the connection to certain cohomology groups of $\sS$-spaces and $\sS$-varieties introduced 
in \cite{spiess2} (see Prop.\ 5.5 and Prop.\ 5.11 in loc.\ cit.). 
If $d\ge 0$ and if $W$ is an admissible $\cO[T_{\fp}]$-module then we have 
\begin{equation}
\label{derham3}
\cH^{\bu}(\Ind_{B_{\fp}}^{G_{\fp}} W)\, \cong \, \bH^{\bu}(X^{\an}; W, \cO)_E
\end{equation}
Furthermore, if $W$ is an $\varpi$-adically admissible $\cO[T_{\fp}]$-module and if $V=W_E$ is the associated Banach space representation of $T_{\fp}$ then
\begin{equation}
\label{derham4}
\cH^{\bu}(\Ind_{B_{\fp}, \cont}^{G_{\fp}} W)\, \cong \, \bH_{\varpi-\ad}^{\bu}(X^{\an}; V, E).
\end{equation}
where the cohomology groups on the right hand side of \eqref{derham3}, \eqref{derham4} have been introduced in (\cite{spiess2}, \S 4.1, 4.3). Here $X^{\an}$ denotes the (complex) quaternionic Hilbert modular $\sS$-manifold $(X_0^D(\fn)^{\fp})^{\an}$ in the sense of (\cite{spiess2}, \S 5.1). Moreover if $d\ge 1$ then we also have 
\begin{equation}
\label{derham5}
\cH^{\bu}(\Ind_{B_{\fp}, \cont}^{G_{\fp}} W)\, \cong \, \bH_{\varpi-\ad}^{\bu}(X_{\barQ}; V, E).
\end{equation}
where now $X= X_0^D(\fn)^{\fp}$ is quaternionic Hilbert modular $\sS$-variety. The cohomology groups $\bH^{\bu}(X^{\an}; W, \cO)_E$ and $\bH_{\varpi-\ad}^{\bu}(X^{\an}; V, E)$ are equipped with a $\Delta$-action and the isomorphism \eqref{derham3}, \eqref{derham4} are Hecke and $\Delta$-equivariant. The groups $\bH_{\varpi-\ad}^{\bu}(X_{\barQ}; V, E)$ are equipped with a Galois action.
If $W$ is an admissible $\cO[T_{\fp}]$-module that is free and of finite rank as an $\cO$-module then $V=W_E$ is an admissible Banach space representation of $T_{\fp}$ and we have (see \cite{spiess2}, \S 4.3)
\begin{equation}
\label{derham6}
\bH^{\bu}(X^{\an}; W, \cO)_E\, \cong \, \bH_{\varpi-\ad}^{\bu}(X_{\barQ}; V, E).
\end{equation}

Note that if $0 \to M_1 \stackrel{\alpha}{ \lra} M_2 \stackrel{\beta}{ \lra} M_3 \to 0$
is a short exact sequence of $\cO[G_{\fp}]$-modules that splits as a sequence of $\cO$-modules then the sequence of $G(F)$-modules
\[
0\lra \cA_{\cO}(M_3, K_0(\fn)^{\fp}; \cO)\lra \cA_{\cO}(M_2, K_0(\fn)^{\fp}; \cO)\lra \cA_{\cO}(M_1, K_0(\fn)^{\fp}; \cO)\lra 0
\]
is exact as well, so we obtain a long exact sequence 
\begin{equation}
\label{extdelta}
\ldots \lra \cH^n(M_3) \stackrel{\beta^*}{ \lra} \cH^n(M_2) \stackrel{\alpha^*}{ \lra} \cH^n(M_1)\stackrel{\delta^n}{\lra} \cH^{n+1}(M_3) \lra \ldots 
\end{equation}
of $\bT_E[\Delta]$-modules. For example the sequence \eqref{extstein9a} yields a long exact sequence
\begin{equation}
\label{extdeltaord}
\ldots \lra \cH^n(\cO(0)) \lra \cH^n(\sE_0) \lra \cH^n(\St_{G_{\fp}})\stackrel{\delta_0^n}{\lra} \cH^{n+1}(\cO(0)) \lra \ldots 
\end{equation}

\begin{prop}
\label{prop:harder}
\noi (a) We have 
\begin{equation}
\label{cohomautom1}
\cH^n(\St_{G_{\fp}, \cont}(\cO))_{\pi}\, =\, \cH^n(\St_{G_{\fp}}(\cO))_{\pi}
\, =\, \left\{\begin{array}{cc} E[\Delta] & \mbox{if $n=d$,}\\
0 &\mbox{otherwise.}\end{array}\right.
\end{equation}
for every $n\ge 0$. In particular we get 
\begin{equation*}
\label{cohomautom2}
\cH^n(\St_{G_{\fp}, \cont}(\cO))_{\pi, \ep}\, =\, \left\{\begin{array}{cc} E & \mbox{if $n=d$,}\\
0 &\mbox{otherwise.}
\end{array}\right.
\end{equation*}
for every character $\ep: \Delta\to \{\pm 1\}$ and $n\ge 0$.

\noi (b) The connecting homomorphisms in \eqref{extdeltaord} induces an isomorphisms
$(\delta_0)_{\pi}: \cH^n(\St_{G_{\fp}}(\cO))_{ \pi} \to \cH^{n+1}(\cO(0))_{\pi}$. Consequently, we have 
\begin{equation*}
\label{cohomautom3}
\cH^n(\cO(0))_{\pi, \ep}\, =\, \left\{\begin{array}{cc} E & \mbox{if $n=d+1$,}\\
0 &\mbox{otherwise}
\end{array}\right.
\end{equation*}
for every character $\ep: \Delta\to \{\pm 1\}$ and $n\ge 0$.

\noi (c) We have
\begin{equation*}
\label{cohomautom4}
\cH^n(\Ind_{B_{\fp}, \cont}^{G_{\fp}}\cO(0))_{\pi, \ep}\, =\, \cH^n(\Ind_{B_{\fp}}^{G_{\fp}}\cO(0))_{\pi, \ep}\, =\,\left\{\begin{array}{cc} E & \mbox{if $n=d, d+1$,}\\
0 &\mbox{otherwise}
\end{array}\right.
\end{equation*}
for every character $\ep: \Delta\to \{\pm 1\}$ and $n\ge 0$.
\end{prop}

\begin{proof} (a) For the first equality \eqref{cohomautom1} note that we have a commutative diagram 
\begin{equation}
\label{picoho1}
\small{\begin{CD}
\ldots \cH^n(\St_{G_{\fp}}(\cO)) @>>> \cH^n(\Ind_{B_{\fp}}^{G_{\fp}}\cO(0)) @>>> \cH^n(\cO(0))@>>> \cH^{n+1}(\St_{G_{\fp}}(\cO)) \ldots \\
@VVV @VV\eqref{derham6}V @VV \id V @VVV\\
\ldots \cH^n(\St_{G_{\fp}, \cont}(\cO)) @>>> \cH^n(\Ind_{B_{\fp}, \cont}^{G_{\fp}}\cO(0)) @>>> \cH^n(\cO(0)) @>>> \cH^{n+1}(\St_{G_{\fp}, \cont}(\cO)) \ldots \\
\end{CD}}
\end{equation}
where the rows are induced by \eqref{steinberg} and \eqref{steinbergcont} respectively. By (\cite{spiess2}, Cor.\ 5.12) and the five lemma all vertical maps in \eqref{picoho1} are isomorphisms. For the second equality in \eqref{cohomautom1} see (\cite{spiess2}, Prop.\ 5.13). 

(b) We will show that $\cH^n(\sE_0)_{\pi} = 0$ for every $n\ge 0$. For that we use the long exact sequence induced by the resolution \eqref{extstein9b} of $\sE_0$ localized at $\pi$
\begin{equation*}
\label{extdeltase0}
\ldots\cH^{n-1}(\cInd^{G}_{K} \cO(0))_{\pi}\lra \cH^n(\sE_0)_{\pi} \lra \cH^n(\cInd^{G}_{K} \cO(0))_{\pi} \lra \cH^n(\cInd^{G}_{K} \cO(0))_{\pi}\ldots 
\end{equation*}
where $K=G(\cO_{\fp})$. Therefore it suffices to see that $\cH^n(\cInd^{G}_{K} \cO(0))_{\pi}=0$ for every $n\ge 0$. By (\cite{spiess2}, Prop.\ 3.38) we have 
\[
\cH^n(\cInd^{G}_{K} \cO(0))\,\cong \, H^n(X_K^{\an}, \cO)_E\,\cong \, H^n(X_K^{\an}, E).
\]
Since the conductor of $\pi$ is $\fp\fn$ and the level of the quaternionic Hilbert modular variety $X_K^{\an}$ is $= K_0(\fn)\subseteq G(\bA_f)$ we have $H^n(X_K^{\an}, E)_{\pi} = 0$.

 (c) follows from (a) and (b) using diagram \eqref{picoho1}.
\end{proof}

\begin{remark}
\label{remark:ramanunjan}
\rm If $D$ is a totally definite division algebra then the connecting homomorphism
\begin{eqnarray*}
\label{cohomautom5}
&& H^0(G(F), \sA_{\bZ}(\St_{G_{\fp}}(\bZ), K_0(\fn)^{\fp}; \bZ))_E \stackrel{\delta_0^0}{\lra} \\
&& \hspace{2cm} \lra H^1(G(F), \sA_{\bZ}(\bZ, K_0(\fn)^{\fp}; \bZ))_E = H^1(G(F), C(G(\bA_f^{\fp})/K_0(\fn)^{\fp}; \bZ))_E\nonumber
\end{eqnarray*} 
is an isomorphism for any field $E$ of characteristic 0 (i.e. without passing to the localization at $\pi$). In fact, in this case we have 
\[
H^0(X_K, E) \, =\, \left\{\begin{array}{cc} C(G(\bA_f/K_0(\fn), E)
 & \mbox{if $n=0$,}\\
0 &\mbox{if $n\ge 1$}
\end{array}\right.
\]
and $T_{\fp}-(q+1)\id\in \End(C(G(\bA_f/K_0(\fn), E))$ is an isomorphism. \enddemo
\end{remark}

Let $\psi: F_{\fp}^*\to \cO$ be a continuous homomorphism and consider the long exact sequence
\begin{equation}
\label{extdeltapsi}
\ldots \lra \cH^n(\cO(0)) \lra \cH^n(\sE(\psi)) \lra \cH^n(\St_{G_{\fp}, \cont}(\cO))\stackrel{\delta_{\psi}}{\lra} \cH^{n+1}(\cO(0)) \lra \ldots 
\end{equation}
induced by the short exact sequence \eqref{extstein9}. Recall (\cite{spiess1}, Def.\ 6.3), (\cite{gehrmann1}, 2.1) the {\it automorphic $\sL$-invariant} $\sL_{\ep}(\pi, \psi)\in E$ is the uniquely determined scalar satisfying
\begin{equation*}
\label{linv}
\sL_{\ep}(\pi, \psi)\cdot (\delta_0)_{\pi, \ep} \, =\, (\delta_{\psi})_{\pi, \ep}: \cH^d(\St_{G_{\fp}, \cont}(\cO))_{\pi, \ep} \to \cH^{d+1}(\cO(0))_{\pi, \ep}.
\end{equation*}
We will show in the next that $\sL_{\ep}(\pi, \psi)$ independent of the choice of the character $\ep: \Delta \to \{\pm 1\}$.

\begin{lemma}
\label{lemma:linvlin}
Let $\ep:\Delta\to \{\pm 1\}$ be a character. We have

\noi (a) For the normalized valuation $v_{\fp}: F_{\fp}\to \bZ\cup\{+ \infty\}$ we have $\sL_{\ep}(\pi, v_{\fp}) =1$.

\noi (b) The map 
\begin{equation}
\label{linv2}
\sL_{\ep}(\pi, \wcdot): \Hom_{\cont}(F_{\fp}^*, \cO) \lra E, \,\, \psi \mapsto \sL_{\ep}(\pi, \psi)
\end{equation}
is a homomorphism of $\cO$-modules. 
\end{lemma}

\begin{proof} (a) follows from $\delta_{v_\fp} = \delta_0$. For (b) note that by Lemma \ref{lemma:extsteinfunc} and standard properties of $\delta$-functors the map 
\begin{equation*}
\Hom_{\cont}(F_{\fp}^*, \cO) \lra  \Hom_{\cO}(\cH^d(\St_{G_{\fp}, \cont}(\cO))_{\pi, \ep}, \cH^{d+1}(\cO(0))_{\pi, \ep}), \, \psi\mapsto \delta_{\psi}
\end{equation*}
is a homomorphism of $\cO$-modules. 
\end{proof}

Since $\Hom_{\cont}(F_{\fp}^*, \cO)\otimes_{\cO} E = \Hom_{\cont}(F_{\fp}^*, E)$, Lemma \ref{lemma:linvlin} implies that \eqref{linv2} extends to an $E$-linear functional
\begin{equation}
\label{linv3}
\sL_{\ep}(\pi, \wcdot): \Hom_{\cont}(F_{\fp}^*, E) \lra E, \,\, \psi \mapsto \sL_{\ep}(\pi, \psi).
\end{equation}

We need the following criterion characterizing the kernel of \eqref{linv2}.

\begin{prop}
\label{prop:linvinlin2} Let $\psi\in \Hom_{\cont}(F_{\fp}^*, \cO)$ and let $\ep: \Delta\to \{\pm 1\}$ be a character. The following conditions are equivalent
\medskip

\noi (i) $\sL_{\ep}(\pi, \psi) = 0$. 
\medskip

\noi (ii) For either $n=d$ or $n=d+1$ the $\wE$-module $\bH_{\varpi-\ad}^n(X^{\an}; \wE(\Theta_{\psi}), E)_{\pi, \ep}$ is free of rank $1$. 
\end{prop}

Here the $\wE$-module structure on $\bH_{\varpi-\ad}^n(X^{\an}; \wE(\Theta_{\psi}), E)$ is induced by the $\wE$-module structure on $\wE(\Theta_{\psi})$.

\begin{proof} Clearly, (i) is equivalent to the vanishing of the homomorphism
\begin{equation}
\label{linv4}
(\delta_{\psi})_{\pi, \ep}: \cH^d(\St_{G_{\fp}, \cont}(\cO))_{\pi, \ep}\lra \cH^{d+1}(\cO(0))_{\pi, \ep}.
\end{equation}
Consider the long exact sequence \eqref{extdelta} associated to \eqref{extstein7}
\begin{eqnarray}
\label{limEmSS4}
& \ldots \lra  \cH^n(\Ind_{B, \cont}^G \cO(0))\stackrel{\beta^*}{\lra}\cH^n(\Ind_{B, \cont}^G \wcO(\Theta_{\psi})) \stackrel{\alpha^*}{\lra} \cH^n(\Ind_{B, \cont}^G \cO(0)) \\
& \stackrel{\delta_{\Theta_{\psi}}}{\lra}\cH^{n+1}(\Ind_{B, \cont}^G \cO(0))  \lra\ldots 
\nonumber
\end{eqnarray}
Prop.\ \ref{prop:harder} (c) together with the fact that $\beta^* \circ \alpha^*$ is multiplication with $\vep$ implies 
\begin{equation}
\label{limEmSS5}
\dim_E \left( \cH^n(\Ind_{B, \cont}^G \wcO(\Theta_{\psi}))_{\pi, \ep}/ \vep \cH^n(\Ind_{B, \cont}^G \wcO(\Theta_{\psi}))_{\pi, \ep} \right) \, =\, 1,
\end{equation}
i.e.\ the $\wE$-module $\cH^n(\Ind_{B, \cont}^G \wcO(\Theta_{\psi}))_{\pi, \ep}$ is generated by one element if $n=d,d+1$. It follows from standard functorial properties of connecting homomorphisms that the diagram
\[
\begin{CD}
\cH^n(\Ind_{B, \cont}^G \cO(0)) @>\delta_{\Theta_{\psi}} >> \cH^{n+1}(\Ind_{B, \cont}^G \cO(0))\\
@AAA @VVV\\
\cH^n(\St_{G_{\fp}, \cont}(\cO))@> \delta_{\psi} >> \cH^{n+1}(\cO(0))
\end{CD}
\]
commutes for each $n\in \bZ$. Here the vertical maps are induced by the homomorphisms appearing in the sequence \eqref{steinbergcont}. If $n=d$ the vertical maps are -- after localization with respect to $(\pi, \ep)$ -- isomorphisms of onedimensional $E$-vector spaces by Prop.\ \ref{prop:harder} (a), (b). Hence the vanishing of the map \eqref{linv4} is equivalent to the vanishing of  
\begin{equation}
\label{connecan}
(\delta_{\Theta_{\psi}})_{\pi, \ep}: \cH^d(\Ind_{B, \cont}^G \cO(0))_{\pi, \ep} \lra \cH^{d+1}(\Ind_{B, \cont}^G \cO(0))_{\pi, \ep}.
\end{equation}
Localizing \eqref{limEmSS4} with respect to $(\pi, \ep)$ and using Prop.\ \ref{prop:harder} (c) this implies that 
\[
\sL_{\ep}(\pi, \psi) = 0 \Leftrightarrow \dim_E \cH^d(\Ind_{B, \cont}^G \wcO(\Theta_{\psi}))_{\pi, \ep} =2 \Leftrightarrow  \dim_E \cH^{d+1}(\Ind_{B, \cont}^G \wcO(\Theta_{\psi}))_{\pi, \ep} =2.
\]
Using \eqref{limEmSS5} and this yields the equivalence of the conditions (i) and (ii).
\end{proof}

\paragraph{Arithmetic $\sL$-invariants} We recall the definition of the arithmetic (i.e.\ Mazur-Tate-Teitelbaum) $\sL$-invariant $\sL(V_{\pi}, \psi)$ (see \cite{mtt}, \cite{grstevens}). It is defined in terms of the $\varpi$-adic Galois representation $V_{\pi}$ attached to $\pi$. Recall (\cite{carayol}, \cite{wiles}, \cite{taylor}) that there exists a twodimensional $E$-vector space $V = V_{\pi}$ together with a continuous homomorphism $\rho= \rho_{\pi}: \fG:=\Gal(\overline{F}/F)\to \GL(V_{\pi})$ with the following properties
\medskip

\noi (i) $\rho$ is unramified outside the set of primes dividing $p\fd\fn$.
\medskip

\noi (ii) If $\Frob_{\fq} \in \Gal(\overline{F}/F)$ is a Frobenius for a prime $\fq\nmid p\fn\fd$ then we have 
\begin{equation}
\label{galrephmf}
\Tr(\rho(\Frob_{\fq})) = \la(T_{\fq}), \qquad \det( \rho(\Frob_{\fq})) \, =\, \Norm(\fq).
\end{equation}
\medskip

\noi (iii) Let $\fG_{\fp} \subseteq \Gal(\barQ/F)$ be the decomposition group of a prime above $\fp$. Then there exists a short exact sequence of $E[\fG_{\fp}]$-modules 
\begin{equation}
\label{tateperiod}
0 \lra E(1) \lra V \lra E(0) \lra 0.
\end{equation}
where $E(m):= E \otimes_{\bZ_p} \bZ_p(m)$, $m \in \bZ$, i.e.\ $E(m)=E$ equipped with $\fG_{\fp}$-action via the $m$-th power of the cyclotomic character $\chi_{\cycl}: \fG_{\fp}\to \Gal(F_{\fp}(\mu_{p^{\infty}})/F_{\fp})\subseteq \bZ_p^*$. The $\fG_{\fp}$-representation $V$ is semistable but not crystalline\footnote{Since $V$ can be realized as a direct summand of $T_p(A)_E$ of an abelian variety $A/F$ with split multiplicative reduction at $\fp$.}. 

By \cite{ribet} (see also \cite{taylor2}, Prop.\ 3.1) the Galois representation $\rho$ is simple and uniquely determined by the properties (i) and (ii). We denote by
\begin{equation*}
\label{tateperiod2}
\xi= \xi(\pi) \in \Ext_{E[\fG_{\fp}]}^1(E(0), E(1)) \,= \,H^1(F_{\fp}, E(1))
\end{equation*}
the class of the extension \eqref{tateperiod}. Here $H^1(F_{\fp}, \wcdot )= H_{\cont}^1(\fG_{\fp}, \wcdot)$ denotes continuous Galois cohomology. Also the local reciprocity map $\rec: F_{\fp}^*\to \Gal(\overline{F_{\fp}}/F_{\fp})^{\ab} \cong \fG_{\fp}^{\ab}$ induces an isomorphism 
\begin{equation*}
\label{reciproc}
H^1(F_{\fp} , E(0)) \, =\, \Hom_{\cont}(\fG_{\fp}, E) \, \lra \, \Hom_{\cont}(F_{\fp}^*, E), \, \varphi\mapsto \varphi\circ \rec.
\end{equation*}
We denote the inverse map by 
\begin{equation*}
\label{recproc2}
\partial: \Hom_{\cont}(F_{\fp}^*, E)\, \lra \,H^1(F_{\fp}, E(0)), \,\, \psi \mapsto \partial(\psi).
\end{equation*}
Following (\cite{grstevens}, Def.\ 3.9) we define

\begin{df}
\label{df:arithmlinv} 
The $\sL$-invariant of $V_{\pi}$ associated to $\psi\in \Hom_{\cont}(F_{\fp}^*, E)$ is the scalar $\sL(V_{\pi}, \psi)\in E$ characterized by
\[
\sL(V_{\pi}, \psi)\, \partial(v_{\fp}) \cup \xi(\pi) \, =\, \partial(\psi) \cup \xi(\pi).
\]
\end{df}

Here the cup-product is Tate's local duality pairing 
\[
H^1(F_{\fp}, E(0))\times H^1(F_{\fp}, E(1))\lra H^2(F_{\fp}, E(1))\cong E.
\]
The fact that $V$ is not crystalline implies $\partial(\ord_{\fp}) \cup \xi(\pi) \ne 0$. Indeed, since the image of $\xi$ under the projection 
$H^1(F_{\fp}, E(1))\to H_{/f}^1(F_{\fp}, E(1)) := H^1(F_{\fp}, E(1))/H_f^1(F_{\fp}, E(1))$ is non-trivial, its cup-product with the $\partial(v_{\fp})\in H_{f}^1(F_{\fp}, E(0))$ is non-trivial as well. 

\begin{remark}
\label{remark:extsteinfunc}
\rm The map
\[
\sL(V_{\pi}, \wcdot): \Hom(F_{\fp}^*, E)\lra E, \,\, \psi \mapsto \sL(V_{\pi}, \psi)
\]
is an $E$-linear functional with $\sL(V_{\pi}, v_{\fp})=1$. \enddemo
\end{remark}

Our main result is

\begin{theo}
\label{theo:linvmttautom}
We have $\sL_{\ep}(\pi, \psi)\, =\, \sL(V_{\pi}, \psi)$ for every $\psi\in \Hom_{\cont}(F_{\fp}^*, E)$ and every character $\ep:\Delta\to \{\pm 1\}$.\footnote{After this work was completed the author was informed that Gehrmann and Rosso have obtained the same result; see \cite{gehrrosso}.}
\end{theo}

\begin{remark}
\label{remark:mainthm}
\rm The equality of $\sL$-invariants appeared as a conjecture in \cite{greenberg} (in the case $D=M_2(F)$ and $F$ has class number 1) and in \cite{gms} (for $D$ a division algebras). Theorem \ref{theo:linvmttautom} has been known previously in the following cases: $F=\bQ$, $D$ a definite division algebra (see \cite{teitelbaum}), $F=\bQ$, $D=M_2(\bQ)$ (by \cite{grstevens}, \cite{darmon} and \cite{dasgupta}), $F=\bQ$, $D$ a definite division algebra (by \cite{dasgreen} and \cite{lrv}) and for $F$ arbitrary, $D=M_2(F)$, $\psi=\log_p\circ \Norm_{F_{\fp}/\bQ_p}$ and $\ep=1$ (by \cite{spiess1}.

Note thought that even in the classical case $F=\bQ$ and $D=M_2(\bQ)$ our proof is new as it does not use big Galois representations associated to Hida families.
\enddemo
\end{remark} 

\section{Proof of the comparison Theorem if $d=0$ and $d= 1$}
\label{section:shimuracurve}

\paragraph{The case $d=0$} The proof of Thm.\ \ref{theo:linvmttautom} in this case can be deduced from $p$-adic uniformization of Shimura curves. If $F=\bQ$ and $\pi$ is defined over $\bQ$ (i.e.\ if $V_{\pi}$ is the Galois representation associated to an elliptic curve $E/\bQ$) then this has been pointed out already in \cite{teitelbaum}. The proof in the general case does not require any new idea. 

Note that $D$ is a totally definite quaternion algebra, $X$ is a 0-dimensional $\sS$-space and $\Delta =1$. Therefore there exists only one automorphic $\sL$-invariant $\sL(\pi, \psi)$ associated with $\pi$ (and $\psi$ and $\fp$) and we can (and will) drop $\ep$ from the notation. 

We fix an archimedian place $v$ of $F$. Let $\barD$ be the quaternion algebra with $\Ram_{\barD} =\Ram_D \setminus \{v\}\cup \{\fp\}$. Let $\barG$ denote the algebraic group corresponding to $\barD^*/F^*$ and fix an Eichler order $\cO_{\barD}$ of level $\fn$ in $\barD$. We choose isomorphism $\barD_\fq\cong D_\fq$ for every nonarchimedean place $\fq\ne \fp$ that respects the local Eichler orders (i.e.\ $\cO_{\barD, \fq}$ is mapped onto $\cO_{D,\fp}$), so that we can (and will) identify 
$G(\bA_f^{\fp})$ the groups $G(\bA_f^{\fp})$ and $\barG(\bA_f^{\fp})$. Let $\barK_v$ be a maximal compact open subgroup of $\barG_v \cong \PGL_2(\bR)$.  

For a $K_f^{\fp} \subseteq G(\bA_f^{\fp})$ we let $Y = Y(K_f^{\fp})/F$ denote the Shimura curve of level $\barG_{\fp} \times K_f^{\fp}$ associated to $\barD$, i.e.\ the associated compact Riemann surface $Y^{\an}=Y(\bC)$ is given by $Y^{\an}= \barG(F) \backslash \left(\barG_v/\barK_v\times \barG(\bA_f^{\fp})/K_f^{\fp}\right)$. We recall the $\fp$-adic uniformization of $Y$ due to \u{C}erednik \cite{cerednik} if $F=\bQ$ and by Boutot-Zink \cite{boutot-zink} for arbitrary $F$. 

Let $\bC_{\fp} = \widehat{\barF}_{\fp} = \bC_p$ be the completion of the algebraic closure of $F_{\fp}$ and let $Y(K_f^{\fp})^{\rig}$ be the rigid analytic $F_{\fp}$-variety associated to $Y(K_f^{\fp})\otimes F_{\fp}$. For $K_f^{\fp}$ sufficiently small there exists a $\Gal(\bC_{\fp}/F_{\fp})$-equivariant isomorphism of rigid analytic $F_{\fp}$-varieties
\begin{equation}
\label{cerednik}
Y(K_f^{\fp})^{\rig}\, =\, G(F) \backslash \left(\cH_{\fp} \times G(\bA_f^{\fp})/K_f^{\fp}\right)
\end{equation}
where $\cH_{\fp} = \bP^1_{/F_{\fp}}\setminus \bP^1(F_{\fp})$ is the $\fp$-adic upper half plane. Thus the curve $Y(K_f^{\fp})^{\rig}$ is a disjoint union of Mumford curves. This implies that 
the Jacobian $J= \Jac(Y(K_f^{\fp}))$ of $Y(K_f^{\fp})$ has split multiplicative reduction at $\fp$. Using the work of Manin and Drinfeld \cite{mandrin} we can deduce from the isomorphism \eqref{cerednik} an explicit description of the rigid analytic $F_{\fp}$-torus $J^{\rig}$ associated to $J\otimes \bF_{\fp}$ hence also of the Tate module
$T_p(J) = H^1(Y(K_f^{\fp})_{\barQ}, \bZ_p(1))$. We use here the reformulation in (\cite{berggehr} and \cite{dasgupta}) of the theorem of Manin-Drinfeld. 

Firstly, note that if $K_f^{\fp}$ is sufficiently small the $\bZ$-module $H_1(G(F), C_c(G(\bA_f)/K_f^{\fp}, \bZ))$ is free of finite rank. Indeed, let $g_1, \ldots, g_h\in G(\bA_f^{\fp})$ be a system of representatives of $G(F)\backslash G(\bA_f^{\fp})/K_f^{\fp}$ and put $\Gamma_i = G(F) \cap g_i K_f^{\fp}$. If $K_f^{\fp}$ is small enough then the groups $\Gamma_1, \ldots, \Gamma_h$ are finitely generated and free hence 
\[
H_1(G(F), C_c(G(\bA_f)/K_f^{\fp}, \bZ)) \, \cong \, \bigoplus_{i=1}^h H_1(\Gamma_i, \bZ)\, \cong\, \bigoplus_{i=1}^h \Gamma_i^{\ab}
\]
is free-abelian of finite rank. Let $\cT/F_{\fp}$ be the split algebraic torus with character group $H_1(G(F), \break C_c(G(\bA_f)/K_f^{\fp}, \bZ)$). Since the map
\begin{equation*}
\label{caph1}
H^1(G(F), C(G(\bA_f)/K_f^{\fp}, F_{\fp}^*)) \lra \Hom(H_1(G(F), C_c(G(\bA_f)/K_f^{\fp}, \bZ), F_{\fp}^*) 
\end{equation*}
induced by the $\cap$-product is an isomorphism we can identify $H^1(G(F), C(G(\bA_f)/K_f^{\fp}, F_{\fp}^*)$ with set of $F_{\fp}$-points of $\cT$. 

The universal extension of the Steinberg representation \eqref{extsteinuniv} induces 
a homomorphism 
\begin{equation}
\label{deltauniv}
j: H^0(G(F), \cA_{\bZ}(\St_{G_{\fp}}(\bZ), K_f^{\fp}, \bZ)) \lra H^1(G(F), C(G(\bA_f)/K_f^{\fp}, F_{\fp}^*))\, =\, \cT(F_{\fp})
\end{equation}
as follows. Firstly, we note that there exists a canonical map 
\begin{eqnarray}
\label{deltauniv2}
&& \cA_{\bZ}(\St_{G_{\fp}}(\bZ), K_f^{\fp}, \bZ) \, =\, \Hom(\St_{G_{\fp}}(\bZ), \Maps(G(\bA_f^{\fp})/K_f^{\fp}, \bZ))\\
&& \lra \, \prolimn  \Hom(\St_{G_{\fp}}(\bZ)\otimes F_{\fp}^*/U_{\fp}^{(n)}, \Maps(G(\bA_f^{\fp})/K_f^{\fp}, \bZ)\otimes F_{\fp}^*/U_{\fp}^{(n)}) \nonumber \\
&& \lra \,  \prolimn \Hom(\St_{G_{\fp}}(F_{\fp}^*/U_{\fp}^{(n)}), \Maps(G(\bA_f^{\fp})/K_f^{\fp}, F_{\fp}^*/U_{\fp}^{(n)})) \nonumber \\
&& \lra \, \Hom(\prolimn \St_{G_{\fp}}(F_{\fp}^*/U_{\fp}^{(n)}), \prolimn \Maps(G(\bA_f^{\fp})/K_f^{\fp}, F_{\fp}^*/U_{\fp}^{(n)})) \nonumber \\
&&  = \, \Hom(\St_{G_{\fp}, \cont}(F_{\fp}^*), C(G(\bA_f^{\fp})/K_f^{\fp}, F_{\fp}^*))\nonumber 
 \end{eqnarray}
(for the second map note that $\St_{G_{\fp}}(\bZ)\otimes A= \St_{G_{\fp}}(A)$ for any discrete abelian group $A$). The extension \eqref{extsteinuniv} yields a short exact sequence
\begin{eqnarray*}
&& 0 \lra C(G(\bA_f^{\fp})/K_f^{\fp}, F_{\fp}^*) \lra \Hom(\sE_{\univ}, C(G(\bA_f^{\fp})/K_f^{\fp}, F_{\fp}^*))\\
&& \hspace{2cm} \lra \Hom(\St_{G_{\fp}, \cont}(F_{\fp}^*), C(G(\bA_f^{\fp})/K_f^{\fp}, F_{\fp}^*))\lra 0
\end{eqnarray*}
and \eqref{deltauniv} is defined as the composite 
\begin{eqnarray}
\label{deltauniv3}
&& H^0(G(F), \cA_{\bZ}(\St_{G_{\fp}}(\bZ), K_f^{\fp}, \bZ)) \lra H^0(G(F), \Hom(\St_{G_{\fp}, \cont}(F_{\fp}^*), C(G(\bA_f^{\fp})/K_f^{\fp}, F_{\fp}^*)))\\
&& \hspace{2cm} \lra H^1(G(F), C(G(\bA_f)/K_f^{\fp}, F_{\fp}^*))\nonumber
\end{eqnarray}
where the first map is induced by \eqref{deltauniv2} and the second is the connecting homomorphism induced by \eqref{deltauniv3}. 

By (\cite{berggehr}, Thm.\ 4.9) and (\cite{dasgupta}, Thm.\ 2.5) we have an isomorphism of rigid analytic $F_{\fp}$-tori
\begin{equation}
\label{mandrin1}
J^{\rig}\, =\, \Jac(Y(K_f^{\fp}))\, \cong \, \cT^{\rig}/j(\Lambda)
\end{equation}
where $\Lambda = H^0(G(F), \cA_{\bZ}(\St_{G_{\fp}}(\bZ), K_f^{\fp}, \bZ))$. 

We fix $\psi\in \Hom_{\cont}(F_{\fp}^*, \cO)$. As in (\cite{grstevens}, \S 3) we can associate an $\sL$-invariant to $\sL(J^{\rig}, \psi)$ to $J^{\rig}$ and $\psi$ as follows. The homomorphisms $v_{\fp}, \psi$ induce homomorphisms 
\[
(v_{\fp})_*, \psi_* : H^1(G(F), C(G(\bA_f)/K_f^{\fp}, F_{\fp}^*))_E \to H^1(G(F), C(G(\bA_f)/K_f^{\fp}, \cO))_E.
\]
The $\sL$-invariant $\sL(J^{\rig}, \psi)$ is defined as the endomorphism 
\begin{equation}
\label{linvabvar}
\sL(J^{\rig}, \psi): H^0(G(F), \cA_{\bZ}(\St_{G_{\fp}}(\bZ), K_f^{\fp}, \bZ))_E \lra H^0(G(F), \cA_{\bZ}(\St_{G_{\fp}}(\bZ), K_f^{\fp}, \bZ))_E
\end{equation}
such that $\psi_* \circ j_E \circ \sL^{\rig}(\psi) = (v_{\fp})_* \circ j_E$. Note that for every $\psi\in \Hom_{\cont}(F_{\fp}^*, \cO)$ the map $\psi_* \circ j_E$ can be identified with the connecting homomorphism $\delta_{\psi}^0$ in \eqref{extdeltapsi}. So by Remark \ref{remark:ramanunjan} the $\sL$-invariant \eqref{linvabvar} can be defined even if $K_f^{\fp}$ is not small enough. In particular for $K_f^{\fp} = K_0(\fn)^{\fp}$ the endomorphism 
\eqref{linvabvar} is Hecke equivariant and localizing at $\pi$ yields 
\[
\sL(J^{\rig}, \psi)_{\pi} \, =\, \sL(\pi, \psi).
\]
On the other hand \eqref{mandrin1} implies that there exists a short exact sequence of $\fG_{\fp}$-module 
\begin{equation}
\label{cerednikext}
0 \lra T_p(\cT) = \La'_E(1) \lra H^1(Y(K_f^{\fp})_{\barQ}, \bZ_p(1))_E \lra \La_E(0)\lra 0
\end{equation}
where $\La' = \Hom(\bG_m, \cT) = H^1(G(F), C(G(\bA_f)/K_f^{\fp}, \bZ))$. Let $\xi\in \Ext_{E[\fG_{\fp}]}^1(\La_E(0), \La'_E(1))$ denote the class of \eqref{cerednikext}. By applying the Yoneda pairing we obtain a map 
\[
\wcdot \cdot \xi: \Lambda \otimes H^1(F_{\fp}, E(0)) = \Ext_{E[\fG_{\fp}]}^1(E(0), \La_E(0)) \lra \Ext_{E[\fG_{\fp}]}^2(E(0), \La'_E(1)) \cong \La'_E.
\]
By (\cite{grstevens}, Thm.\ 3.11) the diagram 
\begin{equation}
\label{lriglgal}
{\xymatrix@+0.5pc{\Lambda_E\ar[dd]_{\sL^{\rig}(\psi)}\ar[drr]^{\la\mapsto (\la \otimes \partial(v_{\fp}))\cdot \xi}\\
&& \Lambda_{E'}\\
\La_E\ar[urr]^{\la \otimes \partial(\psi))\cdot \xi}
}}
\end{equation}
commutes. If $K_0(\fn)^{\fp}$ is sufficiently small the sequence \eqref{cerednikext} localized at $\pi$ can be identified with \eqref{tateperiod}. Hence the commutativity of \eqref{lriglgal} implies 
\[
\sL(\pi, \psi)\, =\, \sL(J^{\rig}, \psi)_{\pi} \, =\,  \sL(V_{\pi}, \psi).
\]
In the general case this equality still holds as can be seen by choosing an open normal subgroup $K_f^{\fp}$ of $K_0(\fn)^{\fp}$ that is sufficiently small and passing in \eqref{cerednikext} and \eqref{lriglgal} to $K_0(\fn)^{\fp}/K_f^{\fp}$-invariants.

\paragraph{The case $d=1$} We first assume that $D$ is a division algebra. The initial step is to show that the $E$-linear functional \eqref{linv3} does not depend on the character $\ep: \Delta\to \{\pm 1\}$. By Lemma \ref{lemma:linvlin} (a) it suffices to see that for $\psi\in \Hom_{\cont}(F_{\fp}^*, \cO)$ the vanishing of $\sL_{\ep}(\pi, \psi)$ does not depend on $\ep$. 
We fix $\psi$ and consider the long exact sequence 
\begin{eqnarray*}
\label{banreppsi}
&&  \ldots \lra \bH_{\varpi-\ad}^n(X_{\barQ}; E(0), E) \lra \bH_{\varpi-\ad}^n(X_{\barQ}; \wE(\Theta_{\psi}), E) \lra \bH_{\varpi-\ad}^n(X_{\barQ}; E(0), E)\\
&& \hspace{3cm} \stackrel{\delta_{\Theta_{\psi}}^n}{\lra}\bH_{\varpi-\ad}^{n+1}(X_{\barQ}; E(0), E)\lra \ldots 
\nonumber
\end{eqnarray*}
associated to the short exact sequence of $0\to E(0)\to \wE(\Theta_{\psi}) \to E(0)\to 0$ of admissible 
$E$-Banach space representations of $T_{\fp}$ (i.e.\ the sequence \eqref{extstein7} tensored with $\otimes_{\cO} E$). By \eqref{derham5} it can be identified with the sequence \eqref{limEmSS4} as a sequence of Hecke modules. 

The localization $\bH_{\varpi-\ad}^n(X_{\barQ}; E(0), E)_{\pi}$ for $n=1,2$ can be identified with the $\fG=\Gal(\barQ/F)$-representation $V:=V_{\pi}$. Indeed, by Prop.\ \ref{prop:harder} and (\cite{spiess2}, Thm.\ 5.26) $\bH_{\varpi-\ad}^n(X_{\barQ}; E(0), E)_{\pi}$ is a twodimensional $E$-vector space equipped with a continuous $E$-linear $\fG$-action unramified outside the set of primes dividing $p\fd\fn$ such that the Eichler-Shimura relations \eqref{galrephmf} hold.
The uniqueness of $V$ thus implies 
\[
\bH_{\varpi-\ad}^1(X_{\barQ}; E(0), E)_{\pi}\, \cong \, V\, \cong \,\bH_{\varpi-\ad}^2(X_{\barQ}; E(0), E)_{\pi}.
\]
By using the fact that the $E[\fG]$-module $V$ is simple we see that 
\begin{equation}
\label{connecet}
(\delta_{\Theta_{\psi}})_{\pi}: \bH_{\varpi-\ad}^1(X_{\barQ}; E(0), E)_{\pi} \lra \bH_{\varpi-\ad}^2(X_{\barQ}; E(0), E)
\end{equation}
is either injective or $=0$. Therefore (in the case $d=1$) the vanishing of \eqref{connecet}, hence also 
of \eqref{connecan} and of $\sL_{\ep}(\pi, \psi)$ is independent of the character $\ep: \Delta\to \{\pm 1\}$ so we will drop $\ep$ from the notation.

Since $\sL(\pi, v_{\fp}) = 1= \sL(V, v_{\fp})$ and $\sL(\pi, \wcdot)$, $\sL(V, \wcdot)$ are $E$-linear functionals to finish the proof it suffices to show that for $\psi\in\Hom_{\cont}(F_{\fp}^*, \cO)$ with $\sL(\pi, \psi)=0$ we also have $\sL(V, \psi)=0$, i.e.\ $\partial(\psi) \cup \xi(V) =0$. Put $\wV = \bH_{\varpi-\ad}^1(X_{\barQ}; \wE(\Theta_{\psi}), E)_{\pi}$, $\wV^0 = \bH_{\varpi-\ad}^1(X_{\barQ}; \wE(\Theta_{\psi}), E)^0_{\pi}$, $\wV^{\et}= \bH_{\varpi-\ad}^1(X_{\barQ}; \wE(\Theta_{\psi}), E)^{\et}_{\pi}$, 
$V^0= \bH_{\varpi-\ad}^1(X_{\barQ}; E(0), E)^0_{\pi}$, $V^{\et}:= \bH_{\varpi-\ad}^1(X_{\barQ}; E(0), E)^{\et}_{\pi}$
so that we obtain a diagram 
\begin{equation}
\label{deformgalrep}
\begin{CD} 
@. 0 @. 0 @. 0 @.\\
@. @VVV @VVV @VVV @.\\
0@>>> V^0 @>>> V @>>> V^{\et} @>>> 0\\
@. @VVV @VVV @VVV @.\\
0@>>> \wV^0 @>>> \wV @>>> \wV^{\et} @>>> 0\\
@. @VVV @VVV @VVV @.\\
0@>>> V^0 @>>> V @>>> V^{\et} @>>> 0\\
@. @VVV @VVV @VVV @.\\
@. 0 @. 0 @. 0 @.\\
\end{CD}
\end{equation}
By (\cite{spiess2}, Thm.\ 5.26) the rows are exact and we have $V^0=E(1)$, $V^{\et} = E(0)$ (i.e.\ the first and third row can be identified with the sequence \eqref{tateperiod}). Moreover $\fG_{\fp}$ acts on the $\wE$-modules $\wV^0$ and $\wV^{\et}$ via the characters $\wchi\cdot \chi_{\cycl}$ and $\wchi^{-1}$ respectively. Here $\wchi: \fG_{\fp} \to \wcO^*$ denotes the character of $\fG_{\fp}$ that corresponds to $\Theta_{\psi}$ under the reciprocity map, i.e.\ we have $\wchi(\sigma)= 1+ (\partial\psi)(\sigma)\ep$ for every $\sigma\in \fG_{\fp}$. 

By Prop.\ \ref{prop:linvinlin2} the $\wE$-module $\wV$ is a free of rank $2$ and the middle column of \eqref{deformgalrep} is exact. We claim that the first and third column is exact as well. If we view them as complexes of $\wE[\fG_{\fp}]$-modules denoted by $C^1_{\bu}$ and $C^3_{\bu}$ then the exactness of the middle column implies 
$H_n(C^3_{\bu})\cong H_{n-1}(C^1_{\bu})$ for every $n$. However the $\fG_{\fp}$-action on $H_n(C^1_{\bu})$ and $H_{n-1}(C^3_{\bu})$ are not compatible: on the first group $\fG_{\fp}$ acts via the character $\wchi\cdot \chi_{\cycl}$ 
and on the second via $\wchi^{-1}$. It follows $H_n(C^3_{\bu})=0= H_{n-1}(C^1_{\bu})$ for every $n$, i.e.\ the rows of  \eqref{deformgalrep} are exact. In particular we see that $\wV^0$ and $\wV^{\et}$ are free of rank 1 as $\wE$-modules hence $\wV^0= \wE(\wchi\cdot \chi_{\cycl})$ and $\wV^0= \wE(\wchi^{-1})$.

By twisting the $\fG_{\fp}$-action of each $\wE[\fG_{\fp}]$-module in \eqref{deformgalrep} with $\wchi^{-1}$ we see that the first and third row remain the same whereas the middle row becomes
\[
0 \lra \wE(1)\lra \wV' := \wV(\wchi^{-1}) \lra \wE(\wchi^{-2})\lra 0.
\]
Now we follow the argument in (\cite{grstevens}, proof of Thm.\ 3.14). Consider the diagram of Galois cohomology groups 
\[
\begin{CD}
@.@. H^0(F_{\fp}, E(0))@.\\
@.@.@VV\delta V@. \\
H^1(F_{\fp}, E(1)) @>>> H^1(F_{\fp}, V) @>> > H^1(F_{\fp}, E(0)) @>> \delta' > H^2(F_{\fp}, E(1)) \\
@VVV@VVV@VV\io_2 V@VV\io_1 V\\
H^1(F_{\fp}, \wE(1)) @>>> H^1(F_{\fp}, \wV') @>>> H^1(F_{\fp}, \wE(0)) @>>> H^2(F_{\fp}, \wE(1)) \\
\end{CD}
\]
induced by diagram \eqref{deformgalrep} (twisted by $\wchi^{-1}$). A simple computation shows that the image of $1\in E=H^0(F_{\fp}, E(0))$ under $\delta$ is equal to $-2 \partial(\psi)$. 

Since $H^2(F_{\fp}, \bQ(1))=\bQ_p$ the map $\io_1$ can be identified with $E \to \wE, x\mapsto x\ep$ hence it is injective. Since $\io_2\circ \delta=0$ we also have $\io_1\circ \delta'\circ \delta =0$ hence $\delta'\circ \delta=0$. It follows 
\[
0\, =\, \delta'\circ \delta(1) \, =\, -2 \,\delta'(\partial(\psi)) \, =\, \pm 2 \,\xi\cup \partial(\psi)
\]
and therefore $\xi\cup \partial(\psi)=0$.

Finally, it remains to consider the case $F=\bQ$ and $D= M_2(F)$. The above proof can be easily adapted to this case using  (\cite{spiess2}, Thm.\ 5.34) instead of (\cite{spiess2}, Thm.\ 5.26) and by noting that
\begin{eqnarray} 
\label{h1modcurve6}
&& \bH_{\varpi-\ad}^1(X_{\barQ}; E(0), E)_{\pi}\, =\, \bH_{\varpi-\ad, !}^1(X_{\barQ}; E(0), E)_{\pi},\\
&& \bH_{\varpi-\ad}^1(X_{\barQ}; \wE(\Theta_{\psi}), E)_{\pi}\, =\, \bH_{\varpi-\ad,!}^1(X_{\barQ}; \wE(\Theta_{\psi}), E)_{\pi}.\nonumber
\end{eqnarray}
Indeed, by (\cite{spiess2}, Prop.\ 5.30) all terms of the sequence 
\[
0 \lra \Hom_{E[T_{\fp}]}(E(0) ,C^1_E)_{\pi} \lra \Hom_{E[T_{\fp}]}(\wE(\Theta_{\psi}) ,C^1_E)_{\pi} \lra \Hom_{E[T_{\fp}]}(E(0) ,C^1_E)_{\pi}
\]
vanish. Therefore by localizing the sequence (see \cite{spiess2}, \S 5.7)
\begin{equation*}
\label{h1modcurve3}
0\lra \bH_{\varpi-\ad, !}^1(X_{\barQ}; A(\chi), E) \lra \bH_{\varpi-\ad}^1(X_{\barQ}; A(\chi), E) \lra \Hom_{E[T_{\fp}]}(A(\chi^{-1}),C^1_E).
\end{equation*}
at $\pi$ for $A(\chi) = E(0)$ and $A(\chi) = \wE(\Theta_{\psi})$ yields \eqref{h1modcurve6}. \enddemo

\section{Jacquet-Langlands functoriality for automorphic $\sL$-invariants}
\label{section:mainthm}

In this section we assume that $d\ge 2$. Up to isomorphism there exists a unique quaternion algebra $\barD$ with 
\begin{equation}
\label{prinunits}
\Ram_{\barD} \, = \, \left\{\begin{array}{cc} \Ram_D\cup \,\Sigma & \mbox{if $d$ is even;}\\
\Ram_D\cup \{\sigma_2, \ldots, \sigma_d\} & \mbox{if $d$ is odd.} 
\end{array}\right.
\end{equation}
Let $\bar{d}$ be the number of archimedean places $v$ that split $\barD$ (thus $\bar{d}\in \{0,1\}$ and $\bar{d}\equiv d \!\!\!\mod 2$).

We denote the algebraic group corresponding to $\barD^*/F^*$ by $\barG$ and put $\barDelta = \barG(F)/\barG(F)_+ \cong \{\pm 1\}^{\bar{d}}$. Let $\cO_{\barD}$ be an Eichler order of level $\fn$ in $\barD$. We choose an isomorphism $\barD_v\cong D_v$ for every nonarchimedean place $v$ that respects the local Eichler orders (i.e.\ $\cO_{\barD, v}$ is mapped onto $\cO_{D,v}$). This allows us to view the level $K_0(\fn)^{\fp}\subseteq G(\bA_f^{\fp})$ as a level in $\barG(\bA_f^{\fp})$ and also to view $\sK$ as a subset of the set of compact open subgroups of $\barG_{\fp}$. The associated $\sS$-space (resp.\ $\sS$-scheme if $\bar{d}= 1$) will be abbreviated by $\barX=X^{\barD}_0(\fn)^{\fp}$. 

By the Jacquet-Langlands correspondence there exists an automorphic representation $\JL(\pi) = \pi' = \otimes_v \, \pi'_v$ such that $\pi'_v\cong \pi_v$ for all places $v$ of $F$ where $\barD_v\cong D_v$ and so that $\pi'_v$ is the trivial representation of $\barG_v$ for the places $v$ where $\barD_v\not\cong D_v$ (i.e.\ for the places in $\Sigma$ resp.\ in $\{\sigma_2, \ldots, \sigma_d\}$ if $\bar{d}=0$ resp.\ if $\bar{d} =1$). The main result in this section is

\begin{theo}
\label{theo:linvjl}
We have $\sL_{\ep}(\pi, \psi) = \sL(\pi', \psi)$ for every $\psi\in \Hom_{\cont}(F_{\fp}^*, \cO)$ and every character $\ep: \Delta\to \{\pm 1\}$.
\end{theo}

In particular we see that the automorphic $\sL$-invariant $\sL_{\ep}(\pi, \psi)$ is independent of the choice of the character $\ep$. In the previous section we have shown that Theorem \ref{theo:linvmttautom} holds for $\pi'$ (since $\bar{d} \le 1$). Since the $\varpi$-adic Galois representation attached to $\pi$ and $\pi'$ are the same this implies $\sL_{\ep}(\pi, \psi)= \sL(V_{\pi}, \psi)$ thus completing the proof of Thm.\ \ref{theo:linvmttautom}.

Theorem \ref{theo:linvjl} follows from the following seemingly weaker assertion 

\begin{lemma}
\label{lemma:linvkey}
Let $\psi\in \Hom_{\cont}(F_{\fp}^*, \cO)$ and let $V= \ker(\psi)\cap U_{\fp}$. Assume that 
$\sL(\pi', \psi) =0$ and $U_{\fp}/V\cong \bZ_p$.
Then we have $\sL_{\ep}(\pi, \psi)=0$ for every character $\ep: \Delta\to \{\pm 1\}$.
\end{lemma}

\begin{proof}[Proof of Thm.\ \ref{theo:linvjl} assuming Lemma \ref{lemma:linvkey}] We fix a  character $\ep: \Delta\to \{\pm 1\}$. Note that the $E$-vector space $\Hom_{\cont}(F_{\fp}^*, E)$ has dimension $r+1$ where $r:= [F_{\fp}:\bQ_p] = \dim(U_{\fp}^{(1)}\otimes_{\bZ_p}\bQ_p)$. We choose a specific basis of $\Hom_{\cont}(F_{\fp}^*, E)$ as follows. Let $V_1, \ldots, V_r$ be closed subgroups of $U_{\fp}$ with $U_{\fp}/V_i\cong \bZ_p$ and so that $V_1\cap\ldots \cap V_r$ is the torsion subgroup of $U_{\fp}$. We can find $\psi_1', \ldots, \psi_r' \in \Hom_{\cont}(F_{\fp}^*, \cO)$ so that $\Ker(\psi_i') \cap U_{\fp}= V_i$ for $i=1, \ldots, r$. Then $\psi_1', \ldots, \psi_r', v_{\fp}$ form a basis of $\Hom_{\cont}(F_{\fp}^*, E)$. If we put $\psi_i = \psi_i' - \sL(\pi', \psi_i')\cdot v_{\fp}$ then we have 
 $\Ker(\psi_i) \cap U_{\fp}= \Ker(\psi_i') \cap U_{\fp}= V_i$ and
 \[
\sL(\pi', \psi_i) \, =\, \sL(\pi', \psi_i') - \sL(\pi', \psi) \cdot \sL(\pi', v_{\fp}) \, =\, 0
\]
for $i=1, \ldots, r$ by Lemma \ref{lemma:linvlin}. Therefore $\psi_1, \ldots, \psi_r, v_{\fp}$ form a basis of $\Hom_{\cont}(F_{\fp}^*, E)$ so that $\sL_{\ep}(\pi, \psi_i) =0 = \sL(\pi', \psi_i)$ for $i=1, \ldots, r$ and $\sL_{\ep}(\pi, v_{\fp}) = \sL(\pi', v_{\fp}) = 1$
by Lemmas \ref{lemma:linvlin} (a) and \ref{lemma:linvkey}. Using the fact that both types of $\sL$-invariants are linear as functions of $\psi$ we obtain $\sL_{\ep}(\pi, \psi) = \sL(\pi', \psi)$ for every $\psi\in\Hom_{\cont}(F_{\fp}^*, E)$.
\end{proof}

The proof of Lemma \ref{lemma:linvkey} requires some preparation. We put $\barT_{\fp} = T_{\fp}/V$ and 
consider the Iwasawa algebra 
\begin{equation}
\label{iwasawa}
\La \, =\,  \cO[\![U_{\fp}/V]\!] \, =\, \prolim_n \cO[U_{\fp}/V U_{\fp}^{(n)}].
\end{equation}
The augmentation ideal of $\La_E$, i.e.\ the kernel of the canonical projection $\aug: \La_E\to E$, will be denote by $\fa$. Since $U_{\fp}/V\cong \bZ_p$ we see that $\La$ is non-canonical isomorphic to the power series ring $\cO[\![T]\!]$ in one variable over $\cO$. In particular we note that $\La_E$ is a principal ideal domain. 

We fix a character $\barep: \barDelta\to \{\pm 1\}$ and consider the following $E[\barT_{\fp}]$-module
\begin{equation}
\label{modforms1}
\cV\, =\, \fOrd_{\varpi-\ad}^{V, \bar{d}}(\barX^{\an}, \cO)_{E, \barep}. 
\end{equation}
By (\cite{spiess2}, Prop.\ 5.19 (b) and Prop.\ 5.20) the vector space $\cV$ is an admissible $E$-Banach space representation of $\barT_{\fp}$ and its dual $\cM =\cD(\cV)\in \Mod_E^{\fgaug}(\barT_{\fp})$ is free and of finite rank as a $\La_E$-module. The $\La_{\cO}(\barT_{\fp})_E$-module $\cV$ is equipped additionally with an action of the 
Hecke algebra $\bT_E$ commuting with the $\La_{\cO}(\barT_{\fp})_E$-action. 
Let 
\begin{equation}
\label{heckeimage}
R\, =\, \Image(\bT_E\otimes_E \La_{\cO}(\barT_{\fp})_E \lra \End_{\Ban_E^{\adm}(\barT_{\fp})}(\cV)).
\end{equation}
If $r_0=\rank_{\La_E} \cM$ then using the duality functor \eqref{pontrajagin4} yields
\[
R\, \subseteq \,\End_{\Ban_E^{\adm}(\barT_{\fp})}(\cV) \, \cong \, \End_{\Mod_{E}^{\fgaug}(\barT_{\fp})}(\cM)\subseteq \End_{\La_E}(\cM) \, \cong  \, \La_E^{r_0^2}.
\]
Hence $R$ is a finite and torsionfree, hence flat $\La_E$-algebra. We put $\fX = \Spec R$ and let $q: \fX\to \Spec \La_E$ be the induced finite flat morphism. We denote by 
\begin{equation}
\label{heckeimage2}
\la: \bT_E \lra R_E 
\end{equation}
the canonical $E$-algebra homomorphism and by 
\begin{equation}
\label{heckeimage3}
\chi: F_{\fp}^* \, \cong \,T_{\fp} \lra \barT_{\fp} \lra R^*
\end{equation}
the canonical (continuous) character. Let $|\fX|$ be the set of closed points of $\fX$. For $x\in |\fX|$ with residue field $k(x)$ we put 
\[
\begin{CD}
\la_x : =\, \la \!\! \mod x:  \bT_E@>\la >> R @> \!\mod x >> k(x)
\end{CD} 
\]
and 
\[
\begin{CD}
\chi_x : =\, \chi \!\! \mod x:  \bT_E@>\chi >> R @> \! \mod x >> k(x)^*.
\end{CD} 
\]

\begin{lemma}
\label{lemma:piinfx}
(a) There exists a point $x_0=\fm_0\in |\fX|$ that is associated to $\pi$, i.e\ we have, $\chi_{x_0} = 1$, $\la_{x_0}=\la_{\pi}$ and $k(x_0) = E$. The point $x_0$ lies above $\fa$ and $\fX\to \Spec \La_E$ is {\'e}tale in $x_0$. Moreover the $R_{\fm_0}$-module $\cM_{\fm_0}$ is free of rank 1.

\noi (b) We have $\cV[\fm_0^2]\, \cong \, \wE[\Theta_{\psi}]$ as $E$-Banach space representations of $\barT_{\fp}$.
\end{lemma}

\begin{proof} (a) Since the anti-equivalence of categories $\cD: \Ban_E^{\adm}(\barT_{\fp})\to \Mod_{E}^{\fgaug}(\barT_{\fp})$ is $\La_{\cO}(\barT_{\fp})_E$-linear we have
\[
\cM/\fa \cM \, =\, \cM\otimes_{\La_E, \aug} E \, \cong \, \cD(\cV[\fa]).
\]
By (\cite{spiess2}, Prop.\ 5.19 (e)) we have 
\begin{eqnarray*}
\fOrd_{\varpi-\ad}^{V, \bar{d}}(\barX^{\an}, \cO)_E[\fa] & = & \Hom_{E[\barT_{\fp}]}(\cInd_{\barT_{\fp}^0}^{\barT_{\fp}} E, \fOrd_{\varpi-\ad}^{V, \bar{d}}(\barX^{\an}, \cO)_E)\\
&= &\bH_{\varpi-\ad}^{\bar{d}}(\barX^{\an}; \cInd_{\barT_{\fp}^0}^{\barT_{\fp}} E, E)\\
&=& \Hom_{E[\barT_{\fp}]}(\cInd_{\barT_{\fp}^0}^{\barT_{\fp}} E, \fOrd_{\varpi-\ad}^{T_{\fp}^0, \bar{d}}(\barX^{\an}, \cO)_E) = H^{\bar{d}}(\barX_{K_0(1)}^{\an}, \cO)^{\ord}_E.
\end{eqnarray*}
Let $\barbD$ be $=\barD$ if $\bar{d}=0$ and let $\barbD$ be the totally ramified incoherent quaternion algebra over $\bA$ (in the sense \cite{yuanzhangzhang}) with ramification set $\Ram_{\barbD} = \Ram_{\barD} \cup \{\sigma_1\} = S_{\infty}\cup \Ram_D$ if $\bar{d}=1$. Following (\cite{yuanzhangzhang}, bottom of page 70) for a subfield  $E$ of $\bC$ we define the set $\cA(\barbD^*/\bA^*, E)$ of automorphic representations of $\barbD^*/\bA^*$ over $E$ as the set isomorphism classes of irreducible representations of $\barbD^*/\bA^*$ such that $\pi\otimes_E \bC$ is a sum of automorphic representations of $\barbD^*/\bA^*$ of weight $0$. As in the proof (\cite{spiess2}, Lemma 3.25) we have a decomposition
\begin{equation}
\label{jacshimuraord}
\cV[\fa]\,=\,H^{\bar{d}}(\barX_{K_0(1)}^{\an}, \cO)^{\ord}_{E, \barep}\, =\, \bigoplus_{\wpi\in \cR} \, M_{\wpi, \fp}^{\ord} \otimes_{\cO_{\wpi}} (\wpi^{\fp})^{K_0(\fn)^{\fp}}
\end{equation} 
where $\cR$ denotes the set of $\wpi\in \cA(\barbD^*/\bA^*, E)$ such that $\wpi_{\fp}^{K(0)}\ne 0 \ne (\wpi^{\fp})^{K_0(\fn)^{\fp}}$
and $M_{\wpi, \fp}^{\ord}\ne 0$. Here -- as in the proof of (\cite{spiess2}, Lemma 3.25) -- $M_{\wpi, \fp}$ denotes a $T_{\fp}^+$-stable $\cO_{\wpi}$-lattice in $\wpi_{\fp}^{K_1(n)}$. Moreover, by (\cite{spiess2}, Prop.\ 2.22) there exists an unramified quasicharacter $\chi_{\wpi}: F_{\fp}^* \to \cO^*$ such that $M_{\wpi, \fp}^{\ord} \cong \cO_{\wpi}(\chi_{\wpi})$ for every $\wpi\in \cR$. Thus we get 
\begin{equation}
\label{jacshimuraord2}
\cV[\fa]\,=\, \bigoplus_{\wpi\in \cR} \, E_{\wpi}(\chi_{\wpi}) \otimes_{E_{\wpi}} (\wpi^{\fp})^{K_0(\fn)^{\fp}}.
\end{equation} 
In particular the action of $\bT_E\otimes_E \La_{\cO}(\barT_{\fp})_E$, hence also of $R$, on $\cV[\fa]$ is semisimple. Since $R/\fa R$ is a finite $\Lambda_E/\fa=E$-algebra we conclude 
\[
\cV[\fa] \, \cong \, \bigoplus_{\fm\in \Spec R/\fa R} \cV[\fm]
\]
(here and in the following we view $\Spec R/\fa R$ as a subset of $|\fX|$). In particular the localization $\cV[\fa]_{\fm}$ is equal to $\cV[\fm]$ for every $\fm\in \Spec R/\fa R$. 
Applying again the duality functor $\cD$ yields
\[
\cM_{\fm}/\fa \cM_{\fm} \,= \, \cD(\cV[\fa])_{\fm} \, =\, \cD(\cV[\fm])\, =\, \cM_{\fm}/\fm \cM_{\fm} \, = \, \cM\otimes_R R/\fm
\]
hence
\begin{equation}
\label{semisimplepi}
\fa \cM_{\fm} \,=\, \fm \cM_{\fm}
\end{equation}
for every $\fm \in \Spec R/\fa R\subset |\fX|$. Since $\ff(\pi') =\ff(\pi) = \fp\fn$ and $\pi'_{\fp} = \pi_{\fp} = \St_{G_{\fp}}(E)$ the automorphic representation $\pi$ corresponds to an element of $\cR$ that (by abuse of notation) will also be denoted by $\pi$. Since $\dim_E((\pi^{\fp})^{K_0(\fn)^{\fp}})=1$ the $R$-submodule 
\[
\{ v\in \cV[\fa]\mid t\cdot v= \la_{\pi}(h)v, \, t \cdot v = v \,\forall \, h\in \bT_E,t\in \barT_{\fp}\}
\] 
of $\cV[\fa]$ is onedimensional, so it is of the form $\cV[\fm_0]$ for a unique ideal $\fm_0 \in \Spec R/\fa R$, i.e.\ a unique point $x_0\in \fX$ lying above $\fa$. It follows $k(x_0) = E$ and $\dim_E(\cM\otimes_R R/\fm_0) = 1$. By Nakayama's Lemma we obtain
\[
\cM_{\fm_0} \cong \, R_{\fm_0}/\Ann_{R_{\fm_0}}(\cM_{\fm_0})\, \cong \, (R/\Ann_R(\cM))_{\fm_0}.
\]
However since $\cM$ is a faithful $R$-module we have $\Ann_R(\cM)=0$, hence $\cM_{\fm_0}$ is a free $R_{\fm_0}$-module of rank $1$. The equality \eqref{semisimplepi} for $\fm =\fm_0$ now yields $\fa R_{\fm_0} = \fm_0 R_{\fm_0}$, i.e.\ $\fX \to \Spec \La_E$ is unramified, hence {\'e}tale in $\fm_0=x_0$.

(b) Firstly, we remark that
\[
\Hom_{E[\barT_{\fp}]}( \wE(\Theta_{\psi}), \fOrd_{\varpi-\ad}^{V, \bar{d}}(\barX^{\an}, \cO)_E)\, \cong\, \bH_{\varpi-\ad}^{\bar{d}}(\barX^{\an}; \wE(\Theta_{\psi}), E)
\]
by (\cite{spiess2}, Prop.\ 3.19 (e)). Together with Prop.\ \ref{prop:linvinlin2} it follows that 
$\Hom_{E[\barT_{\fp}]}( \wE(\Theta_{\psi}), \cV_{\fm_0})$
is a free $\wE$-module of rank 1, so there exists a monomorphism of $E[\barT_{\fp}]$-modules
\begin{equation}
\label{linvdeform}
\wE(\Theta_{\psi}) \lra \cV_{\fm_0}.
\end{equation}
The source carries a canonical $\La_E$-module structure and it is easy to see that \eqref{linvdeform} is $\La_E$-linear and that $\fa^2 \wE(\Theta_{\psi})=0$. It follows that the image of \eqref{linvdeform} is contained in $\cV[\fa^2]_{\fm_0}$. By (a) we have 
\[
\cM_{\fm_0}/\fa^2\cM_{\fm_0} = \cM_{\fm_0}/\fm_0^2\cM_{\fm_0}= \cM/\fm_0^2\cong R/\fm_0^2 \cong \La_E/\fa^2
\]
hence $\cV_{\fm_0}[\fa^2]\, =\, \cV[\fm_0^2]$ and $\dim_E(\cV[\fm_0^2]) =2$. Therefore \eqref{linvdeform} can be viewed as an isomorphism $\wE(\Theta_{\psi}) \lra \cV[\fm_0^2]$. 
\end{proof}

A point $x\in |\fX|$ will be called {\it classical} if $\chi_x\ne 1$ is a quasicharacter and if 
there exists a $\pi(x)\in \cA(\bD^*/\bA^*, E)$
\medskip

\noi (i) The  $\pi(x)$ is defined over $k(x)$, i.e.\ there exists $\iota: E_{\pi(x)}:=\End_E(\pi(x))\cong k(x)$;
\medskip

\noi (ii) The conductor of $\pi(x)^{\fd\fp, \infty}$ divides $\fn$;
\medskip

\noi (iii) $\pi(x)_{\fp}$ is the principal series representation $\pi(x)_{\fp} = \pi(\chi_x^{-1}| \wcdot |_{\fp}^{-1/2}, \chi_x| \wcdot |_{\fp}^{1/2})$;
\medskip

\noi (iv) The Hecke eigenvalue homomorphism $\la_{\pi(x)}: \bT \to E_{\pi(x)}\stackrel{\iota}{\cong} k(x)$ is equal to $=\la_x$.

The set of classical points of $\fX$ will be denoted by $\fX_{\cl}$.

\begin{lemma}
\label{lemma:classicalpoints}
$\fX_{\cl}$ is dense in $\fX$.
 \end{lemma}

\begin{proof} For $n\ge 1$ let $U_n$ be the open subgroups of $T_{\fp}^0$ of index $p^n$ containing $V$. Put $\La_n = \cO[T_{\fp}^0/U_n]$ and $\fa_n :=\Ker(\La_E\to (\La_n)_E)$. Since $\bigcap_{n\ge 1} \fa_n = 0$, the set $(\bigcup_{n \ge 1}\Spec R/\fa_n R)\setminus \Spec R/\fa R$ is dense in $\fX$, so it suffices to see that it consists of classical points. 
Similar to \eqref{jacshimuraord}, \eqref{jacshimuraord2} one shows that for fixed $n\ge 1$ we have
\begin{equation*}
\label{jacshimuraord3}
\cV[\fa_n]\,=\,  H^{\bar{d}}(\barX_{K_{U_n}(n)}^{\an}, \cO)^{\ord}_{E, \barep}\, \cong \, \bigoplus_{\wpi\in \cR} \, E_{\wpi}(\chi_{\wpi}) \otimes_{E_{\wpi}} (\wpi^{\fp})^{K_0(\fn)^{\fp}}
\end{equation*} 
where $\cR$ denotes the set of $\wpi\in \cA(\bD^*/\bA^*, E)$ such that $\wpi_{\fp}^{K_{U_n}(n)}\ne 0 \ne (\wpi^{\fp})^{K_0(\fn)^{\fp}}$ and $M_{\wpi, \fp}^{\ord}\ne 0$. As in the proof of Lemma \ref{lemma:piinfx} (a) this implies that the 
action of $R$ on $\cV[\fa_n]$ is semisimple and that 
\[
\cV[\fa] \, \cong \, \bigoplus_{\fm\in \Spec R/\fa_n R} \cV[\fm].
\]
Let $\fm \in \Spec R/\fa_n R\setminus \Spec R/\fa R$. We have $\cV[\fm]\ne 0$ (since $\Ann_R(\cM)=0$ hence $\cD(\cV[\fm])=\cM/\fm\cM\ne 0$), so there exists a unique\footnote{Since the Hecke eigenvalue homomorphism $\la_{\wpi}$ determines $\wpi$ uniquely.} $\wpi \in \cR$ such that 
\[
\cV[\fm]= E_{\wpi}(\chi_{\wpi}) \otimes_{E_{\wpi}} (\wpi^{\fp})^{K_0(\fn)^{\fp}}=\{ v\in \cV[\fa]\mid t\cdot v= \la_{\wpi}(h)v, \, t \cdot v = \chi_{\wpi}(t)v \,\forall \, h\in \bT_E,\, t\in \barT_{\fp}\}.
\] 
The condition $\fm \not \in \Spec R/\fa R$ implies that the quasicharacter $\chi_{\wpi}$ is ramified hence $\wpi_{\fp}=\Ind_{B_{\fp}}^{G_{\fp}} \chi_{\wpi}^{-1} = \pi(\chi_{\wpi}^{-1}| \wcdot |_{\fp}^{-1/2}, \chi_{\wpi}| \wcdot |_{\fp}^{1/2})$ by (\cite{spiess2}, Prop. 2.22). It follows that $\fm$ is classical.
\end{proof}

We fix a character $\ep:\Delta\to\{\pm 1\}$ and consider the cohomology group $\bH^d(X^{\an}; \cV, E)_{\ep}$. It is a finitely generated augmented $E[\barT_{\fp}]$-module equipped with a $\bT_R= \bT_E\otimes_E R$-action (the $R$-action is induced by the $R$-action on $\cV$). For a $\bT_R$-module $N$ we let $N_{\la}$ be the maximal quotient of $N$ where $\bT_E$ acts via $\la:\bT_E\to R$, i.e.\ $N_{\la} = N\otimes_{\bT_R, \la\otimes \id_R} R$. 

\begin{lemma}
\label{lemma:jacquetlanglandsfam}
Let $\cN \colon \,=\, (\bH_{\varpi-\ad}^d(X^{\an}; \cV, E)_{\ep})_{\la}$.

\noi (a) The $R$-module $\cN$ is finitely generated.

\noi (b) There exists an open neighbourhood $\fU$ of $x_0$ in $\fX$ such that $\fU \hra \fX \to \Spec \La_E$ is {\'e}tale and $\wcN|_{\fU}$ is an invertible $\cO_{\fU}$-module.\footnote{Recall that $\widetilde{M}$ denotes the quasicoherent $\cO_{\fX}$-module associated to an $R$-module $M$.}

\noi (c) The canonical map $(\bH_{\varpi-\ad}^d(X^{\an}; \cV, E)\otimes_R k(x_0))_{\pi, \ep}\to \cN\otimes_R k(x_0)$ is an isomorphism of onedimensional $E$-vector spaces. Moreover we have $\bH_{\varpi-\ad}^{d+1}(X^{\an}; \cV, E)[\fm_0]_{\pi, \ep}=0$. \end{lemma}

\begin{proof} (a) follows immediately from the fact that $\bH_{\varpi-\ad}^d(X^{\an}; \cV, E)$ is finitely generated as a $\La_E$- hence also as an $R$-module. 

(b) By Lemma \ref{lemma:piinfx} (a) there exists an open neighbourhood $\fU\subseteq \fX$ of $x_0$ such that $\fU \hra \fX \to \Spec \La_E$ is {\'e}tale and $\wcM|_{\fU}$ is an invertible $\cO_{\fU}$-module. Note that for $x\in |\fU|$ we have 
$\dim_{k(x)} \cV[x]=\dim_{k(x)} \cD(\cV[x]) =  \dim_{k(x)}(\cM\otimes_R k(x)) =1$, hence $\cV(x) = k(x)(\chi_x)$ as $\barT_{\fp}$-representation.

Since $\fU$ is regular of dimension 1, we have an exact sequence 
\begin{equation}
\label{jlf1}
0\to \bH_{\varpi-\ad}^d(X^{\an}; \cV, E)\otimes_R k(x) \to  \bH_{\varpi-\ad}^d(X^{\an}; \cV[x], E) \to \bH_{\varpi-\ad}^{d+1}(X^{\an}; \cV, E)[x]\to 0
\end{equation}
for every $x\in |\fU|$ (as shown in appendix (A3) of \cite{spiess2}). Since the $R$-module $\bH^{d+1}(X^{\an}; \cV, E)$ is finitely generated we may assume -- after shrinking $\fU$ if necessary -- that $\bH^{d+1}(X^{\an}; \cV, E)[x]=0$ for all $x\in \fU-\{x_0\}$. So by applying the functor $N\mapsto N_{\la}$ to the sequence \eqref{jlf1} we get 
\begin{equation*}
\label{jlf2}
\cN\otimes_R k(x) \, \cong \,  \bH_{\varpi-\ad}^d(X^{\an}; k(x)(\chi_x), E)\otimes_{\bT_E, \la_x} k(x) 
\end{equation*}
for all $x\in |\fU|-\{x_0\}$. Let $x\in \fU_{\cl} = \fU \cap \fX_{\cl}$, put $E'=k(x)$ and let $\cO'$ be the valuation ring of $E'$. We claim that 
\begin{equation}
\label{jlf3}
\dim(\cN\otimes_R k(x)) \, \ge \, 1 
\end{equation}
Indeed, by \eqref{derham6} and (\cite{spiess2}, Lemma 3.45) we have
\begin{equation*}
\bH_{\varpi-\ad}^d(X^{\an}; \cV[x], E)_{\ep} \, = \, \bH_{\cO}^d(X^{\an}; \cO'(\chi_x), \cO)_{E, \ep}\, =\, \\
\bH_{\cO'}^d(X^{\an}; \cO'(\chi_x), \cO')_{E', \ep}.
\end{equation*}
so by applying $\wcdot \otimes_{\bT_{E'}, \la_{\pi(x)}} E'$ we obtain \eqref{jlf3} by (\cite{spiess2}, Prop.\ 5.13).

For the point $x_0$ we have 
\begin{equation}
\label{jlf4}
\dim_{k(x_0)}(\cN\otimes_R k(x_0)) \, \le \, 1.
\end{equation}
For that note that since $\bH_{\varpi-\ad}^d(X^{\an}; \cV[x_0], E)\cong H_{\cO}^d(X^{\an}; \cO(0), \cO)_E$, the $E$-vector space
\begin{equation*}
\label{jlf5}
\bH_{\varpi-\ad}^d(X^{\an}; \cV[x_0], E)_{\pi, \ep} \, =\, (\bH_{\varpi-\ad}^d(X^{\an}; \cV[x_0], E)_{\ep})\otimes_{\bT_E, \la_{\pi}} E
\end{equation*}
is onedimensional by. Since the sequence \eqref{jlf1} for $x=x_0$ remains exact after localizing with respect to $\fm_{\pi} = \Ker(\la_{\pi})$ it follows that
\begin{equation}
\label{jlf6}
\left(\bH_{\varpi-\ad}^d(X^{\an}, \cV, E)\otimes_R k(x_0)\right)_{\pi, \ep}
\end{equation}
is an $E$-vector space of dimension $\le 1$. This implies \eqref{jlf4} since $\cN\otimes_R k(x_0)$ is the maximal semistable quotient of \eqref{jlf6}.

The two inequalities \eqref{jlf3} and \eqref{jlf4} combined with the facts that $x\mapsto \dim_{k(x)} \cN\otimes_R k(x)$ is upper semi-continuous and $\fU_{\cl}$ is dense in $\fU$ yields -- after shrinking $\fU$ if necessary -- that $\dim_{k(x)} (\cN\otimes_R k(x))=1$ for every $x\in \fU$. Hence $\wcN|_{\fU}$ is an invertible $\cO_{\fU}$-module.

For (c) consider the diagram 
\[
\begin{CD}
\left( \bH_{\varpi-\ad}^d(X^{\an}; \cV, E)\otimes_R k(x_0)\right)_{\pi, \ep} @>>> \bH_{\varpi-\ad}^d(X^{\an}; \cV[x], E)_{\pi, \ep}\\
@VVV@VVV\\
\cN\otimes_R k(x_0) @>>> \bH_{\varpi-\ad}^d(X^{\an}, \cV[x], E)\otimes_{\bT_E, \la_{\pi}} E.
\end{CD}
\]
The upper horizontal map is injective with cokernel $\cong \bH_{\varpi-\ad}^{d+1}(X^{\an}; \cV, E)[\fm_0]_{\pi, \ep}$. 
As already seen in the proof of (b) the first vertical map is surjective and the second is an isomorphism. Hence the  first vertical map is injective as well. The fact that $\dim_{k(x_0)}(\cN\otimes_R k(x_0))=1=\dim_E(\bH_{\varpi-\ad}^d(X^{\an}; \cV[x], E)_{\pi, \ep})$ implies that all maps are isomorphisms and $\bH_{\varpi-\ad}^{d+1}(X^{\an}; \cV, E)[\fm_0]_{\pi, \ep}\break =0$.
\end{proof}

\begin{proof}[Proof of Lemma \ref{lemma:linvkey}] There exists an exact sequence
\begin{eqnarray*}
&& 0\lra \left(\bH_{\varpi-\ad}^d(X^{\an}; \cV, E)\otimes_R R/\fm_0^2\right)_{\pi, \ep} \lra  \bH_{\varpi-\ad}^d(X^{\an}; \cV[\fm_0^2], E)_{\pi, \ep}\hspace{2cm}\\
&& \hspace{4cm} \lra \bH_{\varpi-\ad}^{d+1}(X^{\an}; \cV, E)[\fm_0^2]_{\pi, \ep}\lra 0.
\end{eqnarray*}
(see \cite{spiess2}, appendix (A3)). The third group vanishes by Lemma \ref{lemma:jacquetlanglandsfam} (c). Together with Lemma \ref{lemma:piinfx} (b) we obtain
\[
\bH_{\varpi-\ad}^d(X^{\an}; \wE[\Theta_{\psi}], E)_{\pi, \ep}\, \cong\, \left( \bH_{\varpi-\ad}^d(X^{\an}; \cV, E)\otimes_R R/\fm_0^2\right)_{\pi, \ep}.
\]
Thus by Prop.\ \ref{prop:linvinlin2} it suffices to see that the group on the right 
is a free $\wE$-module of rank 1. For that let $a\in \La_E$ be a generator of the prime ideal $\fa$, put $\cN'= \bH_{\varpi-\ad}^d(X^{\an}; \cV, E)$
and consider the diagram 
\[
\begin{CD}
@. (\cN'\otimes_R R/\fm_0)_{\pi, \ep} @> a\cdot >> (\cN'\otimes_R R/\fm_0^2)_{\pi, \ep} @> \pr >> (\cN'\otimes_R R/\fm_0)_{\pi, \ep} @>>> 0\\
@. @VVV @VVV  @VVV @.\\
0@>>> \cN\otimes_R R/\fm_0 @> a\cdot >> \cN\otimes_R R/\fm_0^2 @> \pr >> \cN\otimes_R R/\fm_0 @>>> 0
\end{CD}
\]
where the vertical maps are induced by the canonical projection $\cN' \to \cN'_{\la} =\cN$. 
By Lemma \ref{lemma:piinfx} (a) and Lemma \ref{lemma:jacquetlanglandsfam} (b), (c) the rows are exact and 
the first and third vertical maps are isomorphisms. Hence the middle vertical map is an isomorphism and we get
\[
\left( \bH_{\varpi-\ad}^d(X^{\an}; \cV, E)\otimes_R R/\fm_0^2\right)_{\pi, \ep}\, \cong\, \cN\otimes_R R/\fm_0^2\, \cong \, R/\fm_0^2\, \cong \, \La_E/\fa^2\,\cong\, \wE.
\]
\end{proof}

 \textsc{Fakult{\"a}t f{\"u}r Mathematik, Universit{\"a}t Bielefeld, Germany} 

 \textit{E-mail} \text{mspiess@math.uni-bielefeld.de}

\end{document}